\documentclass[reqno]{amsart}

\usepackage{tikz,amssymb}
\usepackage{hyperref}
\usepackage{mathtools} 

\newtheorem{theorem}{Theorem}[section]

\newtheorem{corollary}[theorem]{Corollary}
\newtheorem{lemma}[theorem]{Lemma}
\newtheorem{proposition}[theorem]{Proposition}

\newtheorem{problem}[theorem]{Problem}

\theoremstyle{definition}
\newtheorem{definition}[theorem]{Definition}
\newtheorem{remark}[theorem]{Remark}
\newtheorem{example}[theorem]{Example}

\numberwithin{equation}{section}
\numberwithin{figure}{section}
\numberwithin{table}{section}

\newcommand{\sig}{\sigma}
\newcommand{\alp}{\alpha}
\newcommand{\bet}{\beta}
\newcommand{\eps}{\varepsilon}
\newcommand{\ist}{\ensuremath{[\sigma,\tau]}}
\newcommand{\ost}{\ensuremath{(\sigma,\tau)}}
\newcommand{\covers}{\rightarrow}
\newcommand{\covby}{\leftarrow}

\newcommand{\len}[1]{|#1|}

\newcommand{\subpermtau}[2]{\tau_{[#1,#2]}}
\newcommand{\subpermbet}[2]{\bet_{[#1,#2]}}
\newcommand{\subperm}[3]{#1_{[#2,#3]}}
\newcommand{\subseqtau}[2]{\tau_{#1}\ldots\tau_{#2}}
\newcommand{\subseqbet}[2]{\bet_{#1}\ldots\bet_{#2}}
\newcommand{\bigP}{\ensuremath{\mathcal{P}}}
\newcommand{\lab}[3]{\lambda_{#1}(#2, #3)}
\newcommand\red{\rho}
\renewcommand\S{{\mathcal S}}
\newcommand\ol{\overline}
\newcommand\E{\mathbb{E}}
\renewcommand\P{\mathbb{P}}
\newcommand\xx{\overline{x}}
\newcommand{\discprob}[1]{\P_{#1}(\ist\mbox{ contains a disconnected subinterval})}
\newcommand\omitt[1]{}
\newcommand{\lep}{\le_{\bigP}}

\title{The structure of the consecutive pattern poset}

\author{Sergi Elizalde}%
\address{Department of Mathematics, Dartmouth College, Hanover, NH 03755, USA} \email{sergi.elizalde@dartmouth.edu}%

\author{Peter R. W. McNamara}
\address{Department of Mathematics, Bucknell University, Lewisburg, PA 17837, USA}
\email{peter.mcnamara@bucknell.edu}

\thanks{Sergi Elizalde was partially supported by grant \#280575 from the Simons Foundation, by grant H98230-14-1-0125 from the NSA, and by the Distinguished Visiting Professor Program at Bucknell University. Peter McNamara was partially supported by grant \#245597 from the Simons Foundation. 
}



\begin{document}
\begin{abstract}
The consecutive pattern poset is the infinite partially ordered set of all permutations where $\sigma\le\tau$ if $\tau$ has a subsequence of adjacent entries in the same relative order as the entries of $\sigma$.  We study the structure of the intervals in this poset from topological, poset-theoretic, and enumerative perspectives. In particular, we prove that all intervals are rank-unimodal and strongly Sperner, and we characterize disconnected and shellable intervals. We also show that most intervals are not shellable and have M\"obius function equal to zero.
\end{abstract}

\maketitle

\tableofcontents

\section{Introduction}\label{sec:intro}

Consecutive patterns in permutations generalize well-studied notions such as descents, ascents, peaks, valleys and runs. 
A permutation $\sigma$ is said to be \emph{contained} in another one $\tau$ as a \emph{consecutive pattern} if $\tau$ has a subsequence of adjacent entries in the same relative order as the entries of $\sigma$.  Otherwise, $\tau$ is said to \emph{avoid} $\sigma$.  
For example, permutations that avoid both $123$ and $321$ as a consecutive pattern are the well-known alternating and reverse-alternating permutations.
A systematic study of enumerative aspects of consecutive patterns started in~\cite{EN03} and, in the last decade, such patterns have become a vibrant area of research; see~\cite{Eli16} for a recent survey.  Underlying all these questions is a partial order $\bigP$ on the set of all permutations, where we define $\sigma \leq \tau$ if $\sigma$ is contained in $\tau$ as a consecutive pattern.  This paper is the first systematic study of intervals in this consecutive pattern poset.  See Figure~\ref{fig:12-213546} for an example of such an interval.

Some questions for consecutive patterns are motivated by the analogous problems for so-called classical patterns, one of the most actively studied topics of combinatorics in the last three decades.  For the definition of containment in this classical case, we remove the restriction that the relevant subsequence of $\tau$ consists of adjacent entries.  
For example, 123 is less than 2314 in the classical pattern poset but not in the consecutive pattern poset.  See \cite{Bon12} for an exposition of some of the most important developments in the study of classical patterns, \cite{Kit11} for a detailed treatment, and \cite{Ste13} for a survey of recent developments.  

Examples of questions for consecutive patterns that are motivated by the classical case include classifying them into equivalence classes determined by the number of permutations avoiding them~\cite{Nak11}, finding generating functions for the distribution of occurrences of a fixed pattern in permutations~\cite{EN03,MR06,EN12}, and determining the asymptotic growth of the number of permutations avoiding a pattern~\cite{Eli13,Per13}.
Consecutive patterns are interesting not only because the answers to the above questions and the techniques used to solve them are usually quite different than for classical patterns, but also because they have important applications to dynamical systems~\cite{AEK08,Eli09}.

While most of the aforementioned work is enumerative, our approach also has a poset-theoretic and topological flavor.  An early inspiration for research from this viewpoint was a question of Wilf \cite{Wil02} asking for the M\"obius function of intervals in the poset defined by classical pattern containment.  This question remains wide open, but has received increasing attention of late \cite{BJJS11,McSt15,SaVa06,Smi14,Smi15,Smi16,StTe10}.  In contrast, the M\"obius function for intervals in $\bigP$, the consecutive pattern poset, has been determined by Bernini, Ferrari and Steingr\'imsson~\cite{BFS11} and by Sagan and Willenbring~\cite{SaWi12}.  This already gives an indication that the consecutive pattern case is more tractable for certain types of questions than the classical case.

The precursor in the classical case to the present work is \cite{McSt15}, where the focus is on classifying disconnected open intervals and showing that certain special intervals are shellable.    We successfully address these same topological questions of disconnectivity and shellability for the consecutive pattern poset $\bigP$.  Furthermore, we consider what poset-theoretic properties might hold for $\bigP$, e.g., we show that all intervals are rank-unimodal, a statement that is just a conjecture in the classical case \cite{McSt15}.  Recall that a finite graded poset of rank $N$ with $a_i$ elements of rank $i$ is called \emph{rank-unimodal} if the sequence $a_0, a_1, \ldots a_N$ is unimodal, meaning that there exists $k$ with $0 \leq k \leq N$ such that $a_0 \leq a_1 \leq \cdots \leq a_k \geq a_{k+1} \geq \cdots \geq a_N$.
Our main results include the following.
\begin{itemize}
\item A simple classification of those open intervals in $\bigP$ that are disconnected (Theorem~\ref{thm:disconnected}).
\item Almost all intervals contain a disconnected subinterval of rank at least 3, in a certain precise sense, and are thus not shellable (Theorem~\ref{thm:aadiscon} and Corollary~\ref{cor:nonshellable}).
\item All other intervals are shellable (Theorem~\ref{thm:shellable}).  One motivation for shellability is that it completely determines the homotopy type of the order complex of the interval as a wedge of spheres.  The number of spheres is the absolute value of the M\"obius function and is thus readily determined using \cite{BFS11,SaWi12}.
\item All intervals are rank-unimodal (Corollary~\ref{cor:rankunimodal}).  The highest rank of the maximum cardinality is easily determined and the interval takes on a particularly simple structure above this rank.
\item All intervals are strongly Sperner (Theorem~\ref{thm:sperner}), a condition equating the sizes of the union of the $k$ largest ranks and the largest union of $k$ antichains, for all $k$.  
\item It is clear from the formula for the M\"obius function $\mu(\sigma,\tau)$ that it depends heavily on the \emph{exterior} of $\tau$, which is the longest permutation that is both a proper prefix and proper suffix of $\tau$.  We initiate a study of the asymptotic behavior of the exterior in Section~\ref{sec:exterior}.  We show that its expected length is bounded between $e-1$ and $e$ (Theorem~\ref{thm:expectedxlength}).
\item Almost all intervals have a zero M\"obius function, in a certain precise sense (Corollary~\ref{cor:mu0}).
\end{itemize}

Taken together, these topological, poset-theoretic and enumerative results give a rather comprehensive picture of the structure of intervals in $\bigP$.

Preliminaries and some initial observations are the content of the next section.  Disconnectivity is the subject of Section 3, shellability is addressed in Section 4, and the properties of rank-unimodality and strongly Sperner are shown in Section 5.  The exterior of a permutation is considered in Section 6, along with several open problems about its behavior.  

\begin{figure}[htbp]
\begin{center}
\begin{tikzpicture}[scale=1.0]
\tikzstyle{every node}=[rectangle, rounded corners=3pt, inner sep=3pt];
\draw (2,0) node[draw] (12) {12};
\draw (0,1) node[draw] (123) {123};
\draw (2,1) node[draw] (213) {213};
\draw (4,1) node[draw] (132) {132};
\draw (0,2) node[draw] (2134) {2134};
\draw (2,2) node[draw] (1243) {1243};
\draw (4,2) node[draw] (1324) {1324};
\draw (1,3) node[draw] (21354) {21354};
\draw (3,3) node[draw] (12435) {12435};
\draw (2,4) node[draw] (213546) {213546};
\draw (12) -- (123) -- (2134) -- (21354) -- (213546) -- (12435) -- (1324) -- (132) -- (12);
\draw (12) -- (213) -- (2134);
\draw (213) -- (1324);
\draw (123) -- (1243) -- (21354);
\draw (132) -- (1243) -- (12435);
\end{tikzpicture}
\caption{The interval $[12,213546]$ in $\bigP$.}
\label{fig:12-213546}
\end{center}
\end{figure}

\section{Preliminaries}\label{sec:prelims}

In this section, we collect together some useful terminology and notation, and make some initial observations.

\subsection{Fundamentals}
Let $\S_n$ denote the set of permutations of $[n]\coloneqq\{1,2,\dots,n\}$. If $\tau\in\S_n$, we denote the length of $\tau$ by $\len{\tau}=n$, and we write $\tau$ in one-line notation as $\tau=\tau_1\tau_2\dots\tau_n$.  Given a sequence $a_1a_2\dots a_k$ of distinct positive integers, define its {\em reduction} $\red(a_1a_2\dots a_k)$ to be the permutation of $\{1,\dots,k\}$ obtained by
replacing the smallest entry with~$1$, the second smallest with~$2$, and so on. For example, $\red(394176) = 263154$.
For $1\le i\le j\le n$, let $\subpermtau{i}{j}=\red(\subseqtau{i}{j})$, and note that $\subpermtau{i}{j}\in\S_{j-i+1}$.  

We can now define our poset of interest.  We say that $\tau$ {\em contains} $\sigma$ {\em as a consecutive pattern} if there exist $1\le i\le j\le n$ such that $\subpermtau{i}{j}=\sigma$. We write $\sigma\le\tau$, and we say that $\subseqtau{i}{j}$ is an {\em occurrence} of $\sigma$. The relation $\le$ defines a partial order on the set $\S=\bigcup_{n\ge1}\S_n$ of all permutations, and we denote by $\bigP$ the corresponding partially ordered set, called the {\em consecutive pattern poset}. The poset $\bigP$ is the main object of study in this paper. See Figure~\ref{fig:12-213546} for an example of an interval in $\bigP$.

A related partial order on $\S$ is obtained by considering classical pattern containment instead, i.e., the case when $\tau$ is said to contain $\sigma$ if any subsequence of $\tau$ (not necessarily in consecutive positions) has reduction equal to $\sigma$. We will refer to the resulting poset as the {\em classical pattern poset}.

In the rest of this paper we will say that $\tau$ {\em contains} $\sigma$ to mean that $\tau$ contains $\sigma$ as a consecutive pattern, that is, $\sigma\le\tau$. If this is not the case, we say that $\tau$ {\em avoids} $\sigma$.  

We next recall some common operations on permutations.  The \emph{reversal} of the permutation $\tau = \tau_1 \tau_2 \ldots \tau_n$ is the permutation $\tau_n \ldots \tau_2 \tau_1$.
The \emph{complement} of $\tau$ is the permutation whose $i$th entry is $n+1-\tau_i$\,.
Note that both the reversal and complement operations preserve the order relation in $\bigP$.  We note, however, that unlike in the classical pattern poset, the inverse operation does not preserve the order relation in $\bigP$; e.g., the relation $132 \le 3142$ is not preserved when taking inverses.

Given $\sigma \in \S_m$ and $\tau \in \S_n$, we define their \emph{direct sum} $\sigma \oplus \tau$ as the concatenation of $\sigma$ and the permutation $\tau^{+m}$ formed by adding $m$ to every entry of $\tau$.  Similarly, the \emph{skew sum} $\sigma \ominus \tau$ is the concatenation of $\sigma^{+n}$ and $\tau$.  For example, $1324 \oplus 21 = 132465$ and $1324 \ominus 21 = 354621$.

To consider disconnectivity in Section~\ref{sec:disconnectivity} and shellability in Section~\ref{sec:shellability}, we need to define a simplicial complex associated with any interval of any poset.  Given an interval $\ist$, its \emph{order complex} $\Delta(\sigma,\tau)$ is the simplicial complex whose faces are the chains of the open interval $\ost$.  For example, the order complex of the interval $[12,213546]$ from Figure~\ref{fig:12-213546} is shown in Figure~\ref{fig:ordercomplex}.  
We will postpone the other necessary background about shellability until Section~\ref{sec:shellability}. We refer the reader to \cite{Wac07} for background on poset topology.

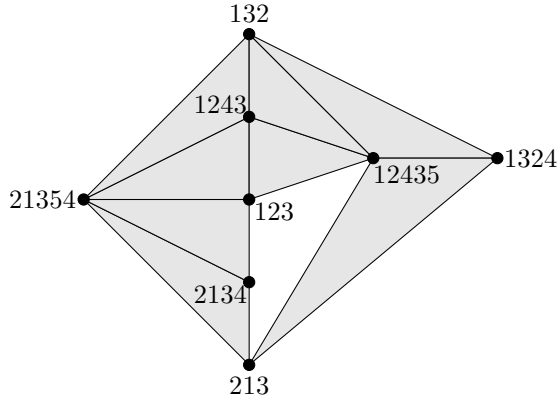
\begin{figure}
\begin{center}
\pgfdeclarelayer{bg}
\pgfsetlayers{bg,main}
\begin{tikzpicture}[scale=0.55]
\tikzstyle{every node}=[circle, inner sep=1.5pt, fill];
\draw (4,4) node[draw] (123) {};
\draw (4,0) node[draw] (213) {};
\draw (4,8) node[draw] (132) {};
\draw (4,2) node[draw] (2134) {};
\draw (4,6) node[draw] (1243) {};
\draw (10,5) node[draw] (1324) {};
\draw (0,4) node[draw] (21354) {};
\draw (7,5) node[draw] (12435) {};

\tikzstyle{every node}=[];
\draw(4,-0.5) node {213};
\draw(-1,4) node {21354};
\draw(3.3,1.7) node {2134};
\draw(4.6,3.7) node {123};
\draw(3.3,6.3) node {1243};
\draw(4,8.5) node {132};
\draw(7.8,4.6) node {12435};
\draw(10.8,5) node {1324};

\begin{pgfonlayer}{bg}
\filldraw[fill=black!10!white, draw=black] (213.center) -- (21354.center) -- (2134.center) -- cycle;
\filldraw[fill=black!10!white, draw=black] (123.center) -- (21354.center) -- (2134.center) -- cycle;
\filldraw[fill=black!10!white, draw=black] (123.center) -- (21354.center) -- (1243.center) -- cycle;
\filldraw[fill=black!10!white, draw=black] (132.center) -- (21354.center) -- (1243.center) -- cycle;
\filldraw[fill=black!10!white, draw=black] (123.center) -- (12435.center) -- (1243.center) -- cycle;
\filldraw[fill=black!10!white, draw=black] (132.center) -- (12435.center) -- (1243.center) -- cycle;
\filldraw[fill=black!10!white, draw=black] (132.center) -- (12435.center) -- (1324.center) -- cycle;
\filldraw[fill=black!10!white, draw=black] (213.center) -- (12435.center) -- (1324.center) -- cycle;
\end{pgfonlayer}{bg}

\end{tikzpicture}
\caption{The order complex of the interval $[12,213546]$ in $\bigP$ from Figure~\ref{fig:12-213546}. Only chains in the \emph{open} interval $\ost$ are included in the order complex of $\ist$; including $\sigma$ and $\tau$ would simply result in the join of the order complex with a line segment and hence would be contractible.}
\label{fig:ordercomplex}
\end{center}
\end{figure}

\subsection{The M\"obius function papers}\label{sub:mobiuspapers}

It remains a wide open problem to determine the M\"obius function for the \emph{classical} pattern poset; see \cite{BJJS11,McSt15,SaVa06,Smi14,Smi15,Smi16,StTe10} for special cases.  In contrast, the M\"obius function for the consecutive pattern poset $\bigP$ has been determined independently by Bernini, Ferrari and Steingr\'imsson in \cite{BFS11} and by Sagan and Willenbring in \cite{SaWi12}.  We will present this recursive formula as stated in the latter paper here because it introduces some key ideas we will need later, and then we will explain the parts of the viewpoint of the former paper that will be useful to us.

 We say that $\sigma$ is a \emph{prefix} (respectively \emph{suffix}) of $\tau \in \S_n$ if $\sig=\subpermtau{1}{i}$ (resp.\ $\sigma=\subpermtau{n-i+1}{n}$) for some $i$, and is a \emph{proper} prefix (resp.\ suffix) if it does not equal $\tau$. For example, $312$ is a proper prefix of $51342$. We say that $\sigma$ is a \emph{bifix} of $\tau$ if it is both a proper prefix and proper suffix of $\tau$.    The \emph{exterior} of $\tau$, denoted $x(\tau)$, is the longest bifix of $\tau$.  Note that $x(\tau)$ always exists whenever $\tau > 1$ since $1$ is a proper prefix and suffix of $\tau$ in that case. The \emph{interior} of $\tau$ is $\subpermtau{2}{n-1}$ and is denoted $i(\tau)$.  For example, if $\tau=21435$, then $x(\tau) = 213$ and $i(\tau) = 132$. We say that $\tau$ is \emph{monotone} if it equals $12\ldots n$ or $n\ldots21$. 
 
\begin{theorem}[\cite{BFS11,SaWi12}]\label{thm:mobius}
For $\sigma, \tau \in \bigP$ with $\sigma \le \tau$, 
\[
\mu(\sigma,\tau) = \left\{
\begin{array}{ll}
\mu(\sigma, x(\tau)) & \mbox{if $|\tau| - |\sigma| > 2$ and $\sigma \le x(\tau) \not\le i(\tau)$}, \\
1 & \mbox{if $|\tau| - |\sigma| = 2$, $\tau$ is not monotone, and $\sigma \in \{i(\tau), x(\tau)\}$}, \\
(-1)^{|\tau| - |\sigma|} & \mbox{if $|\tau| - |\sigma| < 2$},\\
0 & otherwise.
\end{array}
\right.
\]
\end{theorem}

It is clear from this important theorem that the exterior of a permutation is worthy of attention, and yet this attention seems to be absent from the literature.  This motivates our detailed consideration of the statistic $|x(\tau)|$ for general permutations in Subsections~\ref{sub:lengthexterior} and~\ref{sub:asymp_ext}.
In particular, we show that the expected value of $|x(\tau)|$ as $|\tau| \to \infty$ is bounded between $e-1$ and~$e$.  

In \cite{BFS11}, if the exterior $x(\tau)$ satisfies the condition that $\sigma \le x(\tau) \not\le i(\tau)$, then $x(\tau)$ is called the \emph{carrier element} of $\ist$.  As we see in Theorem~\ref{thm:mobius}, when $\ist$ has a carrier element we get the interesting case where the calculation of $\mu(\sigma,\tau)$ reduces to determining $\mu(\sigma,x(\tau))$.  It is therefore natural to ask under what conditions $\ist$ has a carrier element.  This motivates our consideration in Subsection~\ref{sub:carrier} of the permutations $\tau$ such that $[1,\tau]$ has a carrier element.  In this case of $\sigma=1$, $x(\tau)$ can be called the \emph{carrier element of} $\tau$, and it exists if and only if $x(\tau) \not\le i(\tau)$.  We determine the asymptotic probability of a permutation not having a carrier element and, as an application, prove that the M\"obius function is almost always 0, in a particular precise way.  

\subsection{Initial observations}
\ One feature of the consecutive pattern poset that makes it more tractable than the classical pattern poset is that any permutation $\tau$ of length $n$  covers at most two elements, namely $\subpermtau{1}{n-1}$ and $\subpermtau{2}{n}$. The next lemma will be useful later; it appears in \cite{BFS11,SaWi12} and is routine to check.

\begin{lemma}\label{lem:monotone}
For $\tau \in \bigP$ of length $n$, we have $\subpermtau{1}{n-1} = \subpermtau{2}{n}$ if and only if $\tau$ is monotone.
\end{lemma}

As a consequence, we get our first structural result about intervals in $\bigP$.

\begin{proposition}\label{pro:chain}
The interval $[\sigma,\tau]$ in $\bigP$ is a chain if and only if either
\begin{itemize}
\item $\tau$ is monotone, or
\item $\sigma$ occurs exactly once in $\tau$ and does so as a prefix or a suffix.
\end{itemize}
\end{proposition}

\begin{proof}
The ``if'' direction is clear, while for the ``only if'' direction we observe that
if $[\sigma,\tau]$ is a chain, then $\tau$ covers just one element, leading to the two given possibilities. 
\end{proof}

As previously shown in \cite{BFS11}, we can completely determine the structure of $\ist$ when $\sigma$ has just one occurrence in $\tau$, a result we prove here for the purposes of self-containment. This contrasts with the classical pattern poset, where the corresponding result is unknown.

\begin{proposition}[\cite{BFS11}]\label{pro:occursonce}
For an interval $\ist$ in $\bigP$, if $\sigma$ occurs exactly once in $\tau$ then $\ist$ is a product of two chains.  Moreover, if $\sig=\subpermtau{i}{j}$, then these chains have lengths $i$ and $\len{\tau}-j+1$.
\end{proposition}

\begin{proof}
Any $\pi \in \ist$ is obtained from $\tau$ by deleting $a$ entries from the left end of $\tau$ and $b$ entries from the right end of $\tau$, where 
\begin{equation}\label{equ:productofchains}
0\leq a\leq i-1 \mbox{\ \ \ and\ \ \ } 0 \leq b \leq \len{\tau}-j.
\end{equation}
Thus there is a surjective map $\phi$ from such tuples $(a,b)$ to the elements $\pi$ of $\ist$.  This map is injective since there is just one occurrence of $\sigma$ in $\tau$, allowing us to uniquely determine $a$ and $b$ from $\pi$.  We also see that $\phi$ is order-reversing in the sense that $\phi(a,b) \leq \phi(a',b')$ in $\ist$ if and only if $a \geq a'$ and $b \geq b'$ as integers.  Therefore, $\ist$ is order-isomorphic to the dual of a product of two chains, which is itself a product of two chains.  The lengths of the chains follow from \eqref{equ:productofchains}.
\end{proof}

\section{Disconnectivity}\label{sec:disconnectivity}

We say that an interval $\ist$ is \emph{disconnected} (resp.\ \emph{shellable}) if $\Delta(\sigma,\tau)$ is disconnected (resp.\ shellable).  Equivalently, $\ist$ is disconnected if the Hasse diagram of the open interval $\ost$ is disconnected.  For example, in Figure~\ref{fig:12-213546}, the subinterval $[213,213546]$ is disconnected. In examining the structure of intervals of $\bigP$, a natural question is to ask when such intervals $\ist$ are disconnected, preferably giving the answer in terms of simple conditions on $\sigma$ and $\tau$.  We answer this question in Theorem~\ref{thm:disconnected} below.  Determining when an interval contains a non-trivial disconnected subinterval is a more difficult task, which we explore following Theorem~\ref{thm:disconnected}. Interestingly, as we show in Theorem~\ref{thm:aadiscon}, almost all intervals in $\bigP$ do contain such a disconnected subinterval.  

\subsection{Characterization of disconnected intervals}

It will be helpful to deal with posets of rank 2 separately because we can completely classify them, and because they sometimes require special treatment.  In particular, unlike disconnected intervals of rank at least 3, a disconnected interval of rank 2 is shellable; we will study the topic of shellability in detail in the next section.  Since every element of $\bigP$ covers at most 2 elements, a rank-2 interval in $\bigP$ is either a chain or has two elements of rank 1.  Thus the structure of rank-2 intervals is completely determined by Proposition~\ref{pro:chain}.  With rank-2 disconnected intervals now classified, we will follow \cite{McSt15} in saying that a disconnected interval is \emph{non-trivial} if it has rank at least 3.

\begin{definition}
For $\sigma< \tau$, we say that $\sigma$ \emph{straddles} $\tau$ if $\sigma$ is both a prefix and a suffix of $\tau$ and has no other occurrences in $\tau$.
\end{definition}

It is easy to check that $\sigma$ straddles $\tau$ if and only if $\sigma$ is the carrier element of $\ist$, as defined in the last paragraph of Subsection~\ref{sub:mobiuspapers}. Either way we state it, this is exactly the condition that causes non-trivial disconnectedness of $\ist$, as we now show.

\begin{theorem}\label{thm:disconnected}
For $\sigma, \tau \in \bigP$ with $\len{\tau} - \len{\sig} \geq 3$, we have that $\ist$ is disconnected if and only if $\sigma$ straddles $\tau$.
\end{theorem}

\begin{proof}
Let $n=\len{\tau}$. If $\sigma$ straddles $\tau$, then the open interval $\ost$ consists of two chains: one containing the elements $\subpermtau{1}{j}$ for $\len{\sigma}<j<n$,
and one containing elements $\subpermtau{n+1-j}{n}$ for $\len{\sig}<j<n$. These chains are disjoint since $\sigma$ appears as a prefix in the elements of the former chain, and as a suffix in the elements of the latter chain. Thus $\ost$ is disconnected, as required.  This argument is very similar to that of \cite[Lemma~3.3]{BFS11}, where their conclusion is that $\mu(\sigma,\tau)=1$.

To prove the converse, suppose now that $\ost$ is disconnected. In $\bigP$, the permutation $\tau$
covers at most two elements, namely $\subpermtau{1}{n-1}$ and $\subpermtau{2}{n}$.
If these two elements are equal or if one of them avoids $\sigma$, then $\ost$ contains a unique element of length $n-1$, which contradicts the fact that $\ost$ is disconnected.
Thus, we have $\sigma<\subpermtau{1}{n-1}$, $\sigma<\subpermtau{2}{n}$, and $\subpermtau{1}{n-1}\neq\subpermtau{2}{n}$.

In $\bigP$, the permutation $\subpermtau{2}{n-1}$ is covered by both $\subpermtau{1}{n-1}$ and $\subpermtau{2}{n}$. Note that $\subpermtau{2}{n-1}\neq \sigma$, because $\len{\subpermtau{2}{n-1}}=n-2>\len{\sig}$.
If $\subpermtau{2}{n-1}\in\ost$, then the interval $\ost$ would be connected, because from every element there would be a path in the Hasse diagram to either $\subpermtau{1}{n-1}$ or $\subpermtau{2}{n}$, and thus to $\subpermtau{2}{n-1}$.
Therefore, $\subpermtau{2}{n-1}\notin\ost$, which implies that $\subpermtau{2}{n-1}$ avoids $\sigma$. Together with the fact that
$\sigma$ is contained in both $\subpermtau{1}{n-1}$ and $\subpermtau{2}{n}$, it follows that $\sigma$ straddles~$\tau$.
\end{proof}

As we will observe in the next section, an interval is non-shellable if it contains a non-trivial disconnected subinterval.  Therefore the following direct consequence of Theorem~\ref{thm:disconnected} is important to the classification of shellable intervals.

\begin{corollary}\label{cor:disconnected}
An interval $\ist$ contains a non-trivial disconnected subinterval if and only if there exists $\pi \in \ist$ such that there are two adjacent occurrences of $\pi$ in $\tau$ that are offset from each other by at least 3 positions.  Specifically, if $\subseqtau{i}{j}$ is the minimal consecutive subsequence of $\tau$ containing these two adjacent occurrences, then $[\pi, \subpermtau{i}{j}]$ is disconnected.
\end{corollary}

For example, $[1,2143576]$ is not itself disconnected, but the occurrences of $\pi=21$ in $2143576$ start at positions $1$, $3$ and $6$, and we can check that $[\pi,\subpermtau{3}{7}] = [21,21354]$ is disconnected.

In the rest of the paper, we will use the term \emph{disconnected subintervals} to mean non-trivial disconnected subintervals.

\subsection{Finding disconnected intervals}\label{sub:finding_disconnected}
Finding a disconnected subinterval using Corollary~\ref{cor:disconnected} may require checking all possible $\pi\in[\sigma,\tau]$. Our next result implies that, in some cases, it is sufficient to search in the interval $[x(\tau),\tau]$, which may contain significantly fewer elements than $\ist$ (as in Example~\ref{exa:ext_discon} below). 

\begin{proposition}\label{prop:xtautau}
Suppose that $|x(\tau)|\neq2$. If $[1,\tau]$ contains a disconnected subinterval, then so does $[x(\tau),\tau]$.
\end{proposition}

Note that if $[\sigma,\tau]$ contains a disconnected subinterval, then so does $[1,\tau]$, and we can apply the above proposition.  When $\sigma\leq x(\tau)$, this result nicely complements Theorem~\ref{thm:mobius}, which says that computing $\mu(\sigma,\tau)$ often boils down to computing $\mu(\sigma,x(\tau))$. So, when determining the M\"obius function, the interesting part of the interval is often the part below $x(\tau)$, but when looking for disconnected subintervals, the interesting part is often above $x(\tau)$.

\begin{proof}
The statement is trivially true if $x(\tau)=1$, so let us assume that $|x(\tau)|\ge3$.  Let $n=|\tau|$ and $k=|x(\tau)|$.

Suppose that $[x(\tau),\tau]$ contains no disconnected subintervals. Then, by Corollary~\ref{cor:disconnected}, any two adjacent occurrences of $x(\tau)$ in $\tau$ must be offset by one or two positions. If $x(\tau)$ were monotone, then this would force $\tau$ to be monotone as well, but we know that in this case $[1,\tau]$ is a chain, which contains no disconnected subintervals.

Thus, we will assume that $x(\tau)$ is not monotone. By Lemma~\ref{lem:monotone}, no two occurrences of $x(\tau)$ in $\tau$ can be offset by one, so all adjacent occurrences of $x(\tau)$ in $\tau$ must be offset by two positions. Equivalently, $\subpermtau{j}{j+k-1}=x(\tau)$ for every odd $j$ with $1\le j\le n-k+1$, and note that $n$ and $k$ must have the same parity.
This implies that for every $i$, the order relationship between $\tau_i$ and $\tau_{i+1}$ is the same as between $\tau_{i+2}$ and $\tau_{i+3}$. By considering the complement of $\tau$ if necessary, we will assume without loss of generality that $\tau_1<\tau_2$, and thus $\tau_i<\tau_{i+1}$ for every odd $i$. Since $\tau$ is not monotone, we must then have $\tau_i>\tau_{i+1}$ for every even $i$, and so $\tau$ is an up-down permutation, that is, $\tau_1<\tau_2>\tau_3<\tau_4>\dots$.

By Corollary~\ref{cor:disconnected}, the only way for $[1,\tau]$ but not $[x(\tau),\tau]$ to contain a disconnected subinterval would be if there exists a permutation $\pi$ with $x(\tau)\nleq\pi\leq\tau$ which has two adjacent occurrences in $\tau$ offset by 3 or more. Note that we must also have $\pi\nleq x(\tau)$, since otherwise the offsets between adjacent occurrences of $\pi$ would be two. Thus, since $\pi$ is incomparable with $\subpermtau{j}{j+k-1}$ for every odd $j$, we must have that $|\pi|=k$ and any occurrence of $\pi$ in $\tau$ is of the form $\subpermtau{i}{i+k-1}$ for some even $i$. Let $a<b$ be even positions so that $\subpermtau{a}{a+k-1}$ and $\subpermtau{b}{b+k-1}$ are two adjacent occurrences of $\pi$ offset by 3 or more, that is, $\pi$ straddles $\subpermtau{a}{b+k-1}$.

Suppose first that $\tau_1<\tau_3$. Since $k\ge3$, the fact that the patterns $\subpermtau{j}{j+k-1}$ are equal for every odd $j$ implies that $\tau_1<\tau_3<\tau_5<\tau_7<\dots$, that is, the valleys of $\tau$ are increasing.
Let $\sigma=1\oplus\pi$. Since $\tau_{a-1}$ and $\tau_{b-1}$ are valleys, we have that $\subpermtau{a-1}{a+k-1}=\subpermtau{b-1}{b+k-1}=\sigma$. It follows that $\sigma$ straddles $\subpermtau{a-1}{b+k-1}$, since this latter permutation cannot have a third occurrence of $\sigma$ without $\subpermtau{a}{b+k-1}$ having a third occurrence of $\pi$. Additionally, $x(\tau)=\subpermtau{a-1}{a+k-2}\leq\sigma$, since $a-1$ is odd.
Therefore, $[\sigma,\subpermtau{a-1}{b+k-1}]$ is a disconnected subinterval contained in $[x(\tau),\tau]$, contradicting our original assumption.

Now consider the case  $\tau_1>\tau_3$. We now have that the valleys of $\tau$ are decreasing.
If $k$ is odd, then so is $n$, and we can instead consider the reversal of $\tau$, whose first three entries form the pattern $132$, and apply the above argument to it.
If $k$ is even, then $k\ge4$, and so the order relationship between $\tau_2$ and $\tau_4$ determines whether the peaks of $\tau$ are increasing or decreasing. If they are increasing, then $\tau$ is the unique up-down permutation of length $n$ with decreasing valleys and increasing peaks, and $[1,\tau]$ does not have disconnected subintervals since any adjacent occurrences of a pattern $\pi$ will be offset by at most two positions. (That case will be discussed in detail below when studying the permutations with $|x(\tau)|=|\tau|-2$.) If the peaks of $\tau$ are decreasing, then the pattern $\subpermtau{i}{i+k-1}$ is the same 
for every even $i$: it starts with the largest entry $k$, followed by the pattern determined by the first $k-1$ entries of $x(\tau)$. It follows that adjacent occurrences of $\pi$ are offset by 2 and not by 3, which is a contradiction.
\end{proof}

\begin{example}\label{exa:ext_discon}
Consider the interval $\ist = [1, 68372514]$.  We see that $x(\tau)=2413$, so by Proposition~\ref{prop:xtautau}, we only need to look at $[2413,68372514]$ when searching for disconnected intervals in $\ist$.   Using Corollary~\ref{cor:disconnected} and its notation, the only non-trivial candidates for $\pi$ are 2413, 35241 and 52413.  In each case, the adjacent occurrences of $\pi$ in $\tau$ are offset by just two positions, so $\ist$ has no disconnected subintervals.  To use Corollary~\ref{cor:disconnected} without applying Proposition~\ref{prop:xtautau}, we would have had eight non-trivial candidates for $\pi$ to check.  
\end{example}

The condition $|x(\tau)|\neq2$ is used in the proof of Proposition~\ref{prop:xtautau} to argue that all adjacent occurrences of $x(\tau)$ are offset by two positions. To see an example of how the statement may fail when $|x(\tau)|=2$, take $\tau=1325746$. Its exterior is $x(\tau)=12$, but the only disconnected subinterval in $[1,\tau]$ is $[21,21453]$, which is not contained in $[x(\tau),\tau]$.

An immediate consequence of Proposition~\ref{prop:xtautau} is that $[1,\tau]$ has no (non-trivial) disconnected subintervals if $|\tau|- |x(\tau)|\leq 2$. In fact, we can precisely describe the structure of any interval $[\sigma,\tau]$ where $\tau$ has this property, as follows.

When $|\tau|- |x(\tau)|=1$, Lemma~\ref{lem:monotone} implies that $\tau$ is monotone, and in this case $[\sigma,\tau]$ is always a chain. 

When $|\tau|- |x(\tau)|=2$, we have that $x(\tau)=\subpermtau{1}{n-2}=\subpermtau{3}{n}$, where $n=|\tau|$. It follows that 
\begin{equation}\label{eq:shift2}\subpermtau{i}{j}=\subpermtau{i+2}{j+2} \qquad \mbox{for every }1\le i\le j\le n-2.\end{equation}
Thus, for every $2\le k\le n-1$, there are exactly two patterns of length $k$ contained in $\tau$, namely $\subpermtau{1}{k}$ and $\subpermtau{2}{k+1}$. These patterns are different from each other, because otherwise $\tau$ would be monotone, implying that $|\tau|-|x(\tau)|=1$.  For any $\sigma\le\tau$, the interval $[\sigma,\tau]$ has two elements of each length $k$ with $|\sigma|<k<|\tau|$. Additionally, any two elements in $[\sigma,\tau]$ of different lengths are comparable.
The Hasse diagram of such an interval is drawn in Figure~\ref{fig:criss-cross}.  Note that the two elements of length $n-2$ are $x(\tau)$ and $i(\tau)$.
 \begin{figure}[htbp]
\begin{center}
\begin{tikzpicture}[scale=0.8]
\tikzstyle{every node}=[circle, inner sep=1.5pt];
\begin{scope}
\draw (1,0) node[draw] (10) {};
\draw (1,-2ex) node {$\sigma$};
\draw (0,1) node[draw] (01) {};
\draw (2,1) node[draw] (21) {};
\draw (01) -- (10) -- (21);
\draw (01) -- (0,1.4);
\draw (01) -- (0.4,1.2);
\draw (21) -- (2,1.4);
\draw (21) -- (1.6,1.2);
\draw [line width=0.5pt, line cap=round, dash pattern=on 0pt off 3pt] (0,1.4) -- (0,2.2);\draw[densely dotted];
\draw [line width=0.5pt, line cap=round, dash pattern=on 0pt off 3pt] (2,1.4) -- (2,2.2);\draw[densely dotted];
\draw [line width=0.5pt, line cap=round, dash pattern=on 0pt off 3pt] (1,1.4) -- (1,2.2);\draw[densely dotted];
\end{scope}
\begin{scope}[yshift=18ex]
\draw (0,0) node[draw] (00) {};
\draw (2,0) node[draw] (20) {};
\draw (0,1) node[draw] (01upper) {};
\draw (-0.6,1) node {$x(\tau)$};
\draw (2,1) node[draw] (21upper) {};
\draw (2.6,1) node {$i(\tau)$};
\draw (0,2) node[draw] (02) {};
\draw (2,2) node[draw] (22) {};
\draw (1,3) node[draw] (13) {};
\draw(1,3.3) node {$\tau$};
\draw (00) -- (01upper) -- (02) -- (13) -- (22) -- (21upper) -- (20) -- (01upper) -- (22);
\draw (02) -- (21upper) -- (00);
\draw (00) -- (0,-0.4);
\draw (00) -- (0.4,-0.2);
\draw (20) -- (2,-0.4);
\draw (20) -- (1.6,-0.2);
\end{scope}
\end{tikzpicture}
\caption{The interval $\ist$ in the case $|\tau|- |x(\tau)|=2$.}
\label{fig:criss-cross}
\end{center}
\end{figure}

We conclude this section with a result that states that, in a certain precise sense, almost all intervals in $\bigP$ contain disconnected subintervals.  The approach below follows that in \cite{McSt15}, where the analogous result is shown for the classical pattern poset.

\begin{theorem}\label{thm:aadiscon}
Given a permutation $\sigma$, let 
\[
\discprob{n}
\]
denote the probability that $\ist$ contains a non-trivial disconnected subinterval, where $\tau$ is chosen uniformly at random from $\S_n$.  Then
\[
\lim_{n\to\infty} \discprob{n} = 1.
\]
\end{theorem}

Before proving Theorem~\ref{thm:aadiscon}, let us give a quick preliminary result which will also be useful in Subsection~\ref{sub:carrier}.

\begin{lemma}\label{lem:sigletau}
Given a permutation $\sigma$, let $\P_n(\sigma\leq \tau)$ denote the probability that $\sigma\leq \tau$, where $\tau$ is chosen uniformly at random from $\S_n$.  Then
\[
\lim_{n \to \infty} \P_n(\sigma\leq \tau) = 1.
\]
\end{lemma}

\begin{proof}
Let $k = |\sig|$.  Let us break the first $k\lfloor n/k \rfloor$ entries of $\tau$ into $\lfloor n/k \rfloor$ disjoint blocks of length $k$, starting at the beginning.  The probability that none of these blocks is an occurrence of $\sigma$ is 
\[
\left(1-\frac{1}{k!}\right)^{\lfloor n/k \rfloor}.
\]
Since $k$ is fixed, this probability clearly goes to 0 as $n \to \infty$, implying the desired result.\end{proof}

\begin{proof}[\textbf{\bf Proof of Theorem~\ref{thm:aadiscon}}]
Letting $k = |\sig|$, we can assume that $k \geq 3$, since establishing the result for such permutations will also clearly prove it for all $\sigma$.  Assume first that $\sigma_1 > \sigma_k$.  In this case, we see that $\sigma$ straddles $\sigma\oplus \sigma$ and so, by Theorem~\ref{thm:disconnected}, $[\sigma, \sigma\oplus\sig]$ is disconnected.  Applying Lemma~\ref{lem:sigletau} tells us that $\tau$ contains $\sigma\oplus\sigma$ almost surely as $n\to \infty$.  We conclude that as $n \to \infty$, the interval $\ist$ contains the disconnected interval $[\sigma,\sigma\oplus\sig]$ with probability approaching 1, as required.

If $\sigma_1 < \sigma_k$ , then $\sigma$ straddles $\sigma\ominus\sigma$ and we proceed similarly.
\end{proof}

In contrast with Theorem~\ref{thm:aadiscon}, we note that there are infinite classes of intervals that do not contain disconnected subintervals.  For a simple example, any time $\sigma$ satisfies $\sig_1 > \sig_{|\sig|}$, we can take monotone increasing sequences $\alpha$ and $\beta$ and let $\tau =\alpha \oplus \sigma\oplus \beta$.  Then any element of $\ist$ will occur only once in $\tau$, so there are no disconnected subintervals.  For examples of intervals $\ist$ that are not just a product of two chains but contain no disconnected subintervals, let $\sigma= 21$ and $\tau = \alpha \oplus 21\oplus 21\oplus \cdots \oplus 21 \oplus \beta$, where the number of copies of 21 is at least 2 and where $\alpha$ and $\beta$ are monotone increasing as before.

\section{Shellability}\label{sec:shellability}

Our goal for this section is to prove that all intervals in $\bigP$ are shellable, except for those which are not shellable for a straightforward reason.  We begin by giving the necessary background on shellability.  We refer the interested reader to \cite{Wac07} for further details and a wealth of other information about poset topology.

\subsection{Background on shellability and CL-shellability}\label{sub:clbackground}
When considering a combinatorially defined simplicial complex, such as the order complex $\Delta(\sigma,\tau)$ of an interval $\ist$ of $\bigP$, it is common to ask if the simplicial complex is shellable.  Since our intervals $\ist$ are graded, we can restrict our discussion to order complexes that are pure, i.e., all the facets have the same dimension.   A \emph{shelling} of a pure $d$-dimensional simplicial complex $\Delta$ is a linear ordering $F_1, F_2, \ldots,F_s$ of its facets such that the intersection 
\begin{equation}\label{equ:shelling}
\left( \bigcup_{j=1}^{k-1} F_j\right) \cap F_k
\end{equation}
is pure and $(d-1)$-dimensional for all $k=2,3,\ldots,s$.  A simplicial complex is \emph{shellable} if it has a shelling.  Figure~\ref{fig:ordercomplex} is an example of a non-shellable complex since in any ordering of the facets, the point 213 will be a $(d-2)$-dimensional connected component of an intersection of the type in \eqref{equ:shelling}, implying that no shelling exists.  As is customary, we will say that $\ist$ is \emph{shellable} if $\Delta(\sigma,\tau)$ is.  

Much of the the interest in shellability arises from the fact that if an interval $\ist$ is shellable, then this tells us that the homotopy type of $\Delta(\sigma,\tau)$ is a wedge of spheres of the top dimension.  The number of spheres is $|\mu(\sigma,\tau)|$ which, in the case of intervals in $\bigP$, can be determined from Theorem~\ref{thm:mobius}.  In the shellable case, we deduce that either the number of spheres is one or $\Delta(\sigma,\tau)$ is contractible.  

These topological considerations are also why we take the permutation 1 as the bottom element of $\bigP$, rather than the empty permutation $\emptyset$: since $\emptyset$ is only covered by the permutation 1, any complex of the form $\Delta(\emptyset,\pi)$ will be contractible.

We will prove shellability using Bj\"orner and Wachs's theory of CL-shellability, which is a generalization\footnote{Actually, it remains open whether there exist posets that are CL-shellable but not EL-shellable.  The point of CL-shellability is that it is more flexible than EL-shellability, making CL-shellings sometimes easier to find.} of Bj\"orner's theory of EL-shellability \cite{Bjo80}.  The original definition of CL-labeling was given in \cite{BjWa82} under the name ``L-labeling''; the now customary definition below appears in \cite{BjWa83}.  We will partially follow \cite{Wac07} in the exposition here. 

A poset $Q$ is said to be \emph{bounded} if it has a unique minimal and a unique maximal element, which we denote $\hat{0}$ and $\hat{1}$ respectively.  When $x$ is covered by $y$, we will write the corresponding edge as $x \covby y$\,; the head of the arrow reflects the $<$ symbol, and this notation will seem natural later when we primarily follow chains from top to bottom.  Rather than just labeling edges $x \covby y$ as is the case for EL-labelings, we consider the set $\mathcal{ME}(Q)$ of pairs $(C, x \covby y)$ consisting of a maximal chain $C$ and an edge along that chain.  A \emph{chain-edge labeling} of $Q$ is a map $\lambda:\mathcal{ME}(Q) \to \Lambda$, where $\Lambda$ is some poset, satisfying the following condition: if two maximal chains coincide along their bottom $k$ edges, then their labels also coincide along these edges.  The key point is that the label on an edge depends on the maximal chain along which we arrive at that edge.  See Figure~\ref{fig:cl} for an example.

It follows that to restrict a chain-edge labeling of $Q$ to an interval $[u,v]$ in $Q$, we need to record how we arrived at the bottom element $u$ of the interval.  Therefore, we introduce the idea of a \emph{bottom-rooted interval} $[u,v]_r$, which is the interval $[u,v]$ together with a maximal chain $r$ of $[\hat{0}, u]$.  This rooting allows the restriction of $\lambda$ to an interval $[u,v]$ to be consistent with $\lambda$ defined over all of $Q$.  More precisely, for an edge $x \covby y$ in $[u,v]_r$ the label received by $x \covby y$ when considered as being along the maximal chain $C$ of $[u,v]_r$ is equal to the label received by $x \covby y$ when considered as begin along the maximal chain $r \cup C \cup s$ of $Q$, where $s$ is any maximal chain of $[v,\hat{1}]$.

When we say that a maximal chain $M$ of $[u,v]_r$ is \emph{weakly increasing} or \emph{lexicographically precedes} another, we are referring to the sequence of labels along $M$ as we read from $u$ up to~$v$. 

\begin{definition}
Let $Q$ be a bounded poset.  A chain-edge labeling $\lambda$ is a \emph{CL-labeling} (chain-lexicographic labeling) if in each bottom-rooted interval $[u,v]_r$ of $Q$, there is a unique weakly increasing maximal chain, and this chain lexicographically precedes all other maximal chains in $[u,v]_r$\,.  A poset which has a CL-labeling is said to be \emph{CL-shellable}.
\end{definition}

\begin{figure}
\begin{center}
\begin{tikzpicture}[scale=0.55]
\tikzstyle{every node}=[circle, inner sep=0pt];
\draw (2,0) node[draw,minimum size=3ex] (a) {$a$};
\draw (0,2) node[draw,minimum size=3ex] (b) {$b$};
\draw (4,2) node[draw,minimum size=3ex] (c) {$c$};
\draw (0,5) node[draw,minimum size=3ex] (d) {$d$};
\draw (4,5) node[draw,minimum size=3ex] (e) {$e$};
\draw (2,7) node[draw,minimum size=3ex] (f) {$f$};
\tikzstyle{every node}=[text=blue];
\draw (a) --node[left=1pt]{$\textbf{1}$} (b) --node[left=-2pt]{\textbf{2}} (d) --node[left=1.5pt,pos=0.7]{\textbf{2}} (f);
\tikzstyle{every node}=[text=black];
\draw (f) --node[right=1.5pt]{\textbf{2}} (e) -- node[right=2pt, very near end]{\textbf{3}} (b);
\draw (e) --node[right=-1pt]{\textbf{1}} (c);
\tikzstyle{every node}=[text=red];
\draw (a) --node[right=1.5pt]{\textbf{3}} (c) -- node[right=2.5pt, very near end]{\textbf{2}} (d); 
\draw (1.3,5.9) node {\textcolor{red}{$\textbf{1}$}};
\end{tikzpicture}
\caption{A chain-edge labeling that is a CL-labeling.  Notice that the edge $d\covby f$ receives two possible labels: the label 2 as an element of the chain $a \covby b \covby d \covby f$, and the label 1 as an element of the chain $a \covby c \covby d \covby f$.  The bottom-rooted interval $[b,f]_{a \covby b}$ has a unique increasing chain, namely $b \covby d \covby f$ which has label sequence $(2,2)$.}
\label{fig:cl} 
\end{center}
\end{figure}
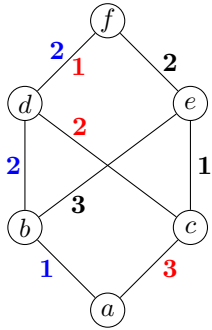

An example of a CL-labeling is shown in Figure~\ref{fig:cl}.    As the terminology suggests, a bounded poset that is CL-shellable is shellable \cite[Theorem~3.3]{BjWa82} and, along with EL-shellability, exhibiting CL-shellability is one of the most common ways to prove shellability of order complexes.

\subsection{Shellability in the consecutive pattern poset}

It is easy to see that disconnected intervals of rank at least 3 are not shellable.   Moreover, a result of Bj\"orner \cite{Bjo80} states that if a poset is shellable, then all of its subintervals are shellable.  In particular, if an interval $\ist$ contains a disconnected subinterval of rank at least 3, then it is certainly not shellable. As an example, the disconnected subinterval $[213, 213546]$ in Figure~\ref{fig:12-213546} results in the non-shellability of the corresponding order complex in Figure~\ref{fig:ordercomplex}.

It thus follows immediately from Theorem~\ref{thm:aadiscon} that almost all intervals in $\bigP$ are not shellable, in a particular precise sense.

\begin{corollary}\label{cor:nonshellable}
Given a permutation $\sigma$, let $\tau$ be chosen uniformly at random from $\S_n$.  Then the probability that $[\sigma,\tau]$ is shellable tends to 0 as $n$ tends to infinity.
\end{corollary}

Our main result of this section is that the converse of Bj\"orner's result is true in $\bigP$: if an interval $\ist$ does not contain a disconnected subinterval, then it is shellable.  So, roughly speaking, all intervals in the consecutive pattern poset that have any hope of being shellable are in fact shellable.  Recall that we have already classified those intervals that contain disconnected subintervals in Corollary~\ref{cor:disconnected} in terms of the entries of $\sigma$ and $\tau$, so the theorem we are about to state classifies those $\sigma$ and $\tau$ that yield a shellable interval $\ist$.

\begin{theorem}\label{thm:shellable}
The interval $\ist$ in $\bigP$ is shellable if and only if it contains no non-trivial disconnected subintervals.
\end{theorem}

It remains to prove the ``if'' direction, which we will do using CL-shellability.  Because we will read the labels on our maximal chains from top to bottom, we will actually show that the dual of $\ist$ is CL-shellable and hence shellable; this implies the shellability of $\ist$ since the order complex of $\ist$ is clearly isomorphic to that of its dual.  The analogue of Theorem~\ref{thm:shellable} in the classical pattern poset is actually false ($[123, 3416725]$ is given in \cite{McSt15} as a counterexample), but we should note that some of the ideas of our proof are similar to those in \cite{McSt15} for proving the dual CL-shellability in the classical case of intervals of layered permutations without non-trivial disconnected subintervals.

Before giving the proof, we will define the chain-edge labeling that we will use.  We will first set provisional labels using a first pass, and then make some modifications in a second pass.  If $\pi'$ has length $\ell$ and covers $\pi$ along some maximal chain $C$ of $\ist$ and $\pi = \subperm{\pi'}{2}{\ell}$, then we label the edge $\pi' \covers \pi$ by $0$.  Otherwise, we assign the label 1.  Referring to Lemma~\ref{lem:monotone}, note that we assign the label 0 whenever $\pi'$ is monotone.  However, the resulting labeling, which is actually an edge-labeling (i.e., each edge label is independent of the chain $C$ on which the edge appears), is not a dual CL-labeling.  For example, the disconnected rank-2 intervals $[21,2143]$ and $[213, 21435]$ each have two increasing chains from top to bottom, one with labels (0,0) and the other with labels (1,1).  Here and elsewhere, we say ``increasing''  to mean weakly increasing.  We next describe the modification we make to our labeling to address this shortcoming.

Pick $\eps > 0$ to be small; $\eps \leq 1/(|\tau|-|\sig|)$ will suffice.  For every maximal chain $C$, work along $C$ from top to bottom, pausing at any triple $\pi'' \covers \pi' \covers \pi$ such that $\pi$ straddles $\pi''$, as in the examples of the previous paragraph.  We see that this situation of straddling for rank-2 intervals is exactly the condition that results in chains with labels (0,0) and (1,1).  Do nothing to the two labels that are already 0.  For the other chain,
letting $\lab{C}{\pi'}{\pi}$ denote the label of the edge $\pi' \to \pi$ along $C$,
decrease the label on $\pi'\covers \pi$ so that $\lab{C}{\pi'}{\pi} = \lab{C}{\pi''}{\pi'} - \eps$.  After completion of these modifications, every label will either be 0 or $1-k\eps$ for some $k \geq 0$.
See Figure~\ref{fig:labeling} for an example.  The edge $213 \covers 21$ there shows that this labeling is a chain-edge labeling rather than just an edge labeling.

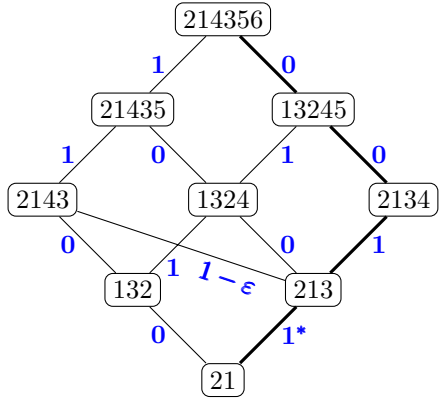
\begin{figure}[htbp]
\begin{center}
\begin{tikzpicture}[scale=1.2]
\tikzstyle{every node}=[rectangle, rounded corners=3pt, inner sep=3pt];
\draw (2,0) node[draw] (21) {21};
\draw (3,1) node[draw] (213) {213};
\draw (1,1) node[draw] (132) {132};
\draw (4,2) node[draw] (2134) {2134};
\draw (2,2) node[draw] (1324) {1324};
\draw (0,2) node[draw] (2143) {2143};
\draw (3,3) node[draw] (13245) {13245};
\draw (1,3) node[draw] (21435) {21435};
\draw (2,4) node[draw] (214356) {214356};
\tikzstyle{every node}=[text=blue];
\draw[very thick] (21) --node[right=1pt]{$\mathbf{1}^{\pmb{*}}$} (213) --node[right=1pt]{\textbf{1}} (2134) --node[right=1pt]{\textbf{0}} (13245) --node[right=1pt]{\textbf{0}} (214356);
\draw (1.45,1.25) node{\textbf{1}};
\draw (21) --node[left=1pt]{\textbf{0}} (132)
(213) --node[right=1pt]{\textbf{0}} (1324) -- (132) --node[left=1pt] {\textbf{0}} (2143) --node[left=1pt]{\textbf{1}} (21435) --node[left=1pt]{\textbf{0}} (1324) --node[right=1pt]{\textbf{1}} (13245)
(21435) --node[left=1pt]{\textbf{1}} (214356)
(2143) --node[sloped, below=-1pt, near end]{$\pmb{1-\eps}$} (213);
\end{tikzpicture}

\caption{The interval $[21,214356]$ with its increasing chain shown in bold. The top-rooted intervals $[213,21435]_{214356 \rightarrow 21435}$ and $[21,2143]_{214356 \rightarrow 21435 \rightarrow 2143}$ have had their labels modified in the particular way we describe.  The edge labeled $1^*$ takes the label $1-2\eps$ along the chain that arrives via 2143 and otherwise takes the label $1$.  All other labels are independent of the chain on which they appear.}
\label{fig:labeling}
\end{center}
\end{figure}

\begin{proof}[\textbf{\bf Proof of Theorem~\ref{thm:shellable}}]
As already mentioned in the first paragraph of this subsection, the interval $\ist$ is not shellable if it contains a non-trivial disconnected subinterval.  For the converse, we will prove dual CL-shellability. So suppose $\ist$ contains no non-trivial disconnected subintervals, and let $[\alp, \bet]_r$ be a \emph{top-rooted interval} in $\ist$, the obvious analogue of bottom-rooted interval.  Using the labeling $\lambda_C$ defined above and reading labels along chains from top to bottom, we wish to show that there is a unique increasing maximal chain in $[\alp,\bet]_r$ and that this increasing chain has the lexicographically first labels of all maximal chains in $[\alp,\bet]_r$\,. 

Every maximal chain $D$ from $\bet$ to $\alp$ can be said to ``end'' at a particular occurrence of $\alp$ in $\bet$ as follows.  Each edge of $D$ corresponds to deleting an entry from the left or from the right of $\bet$, with the convention that the outgoing edge of a monotone permutation corresponds to deleting from the left (this is consistent with the label 0 assigned in our definition of $\lambda_C$);  the entries not deleted from $\bet$ after traversing $D$ give an occurrence of $\alp$.

Define a particular maximal chain $D$ of $[\alp,\bet]_r$ in the following way: starting at $\bet$, delete the leftmost entry if the result will still be in $[\alp,\bet]_r$\,, and otherwise delete the rightmost entry.  In other words, $D$ will delete from the left of $\bet$ until it reaches the rightmost occurrence of $\alp$ in $\bet$, at which point it will switch to deleting from the right.  None of the resulting labels will be of the form $1-k\eps$ for $k>0$ since $D$ will never delete from the right when there is also the option of deleting from the left; see also the technical note in the next paragraph. Thus the sequence of labels will take the form $0, 0, \ldots, 0, 1, 1, \ldots, 1$ from top to bottom, possibly with no 0's or no 1's.  By definition of $D$, these labels are the lexicographically least of any maximal chain.  

It is necessary to make one technical note at this point.  One might wonder about whether the deletions along $r$ could cause any of the labels along $D$ to be of the form $1-k\eps$ with $k>0$.  Let us show that this never happens.  We see that the only possibility is that the first label within $[\alp,\bet]_r$ along $D$ is of this form.  By definition of $D$, this implies that $D$ always deletes from the right, and so $\alp$ only occurs in $\bet$ as a prefix.  Thus there is only one maximal chain $D$ in $[\alp,\bet]_r$ and it has label sequence $(1-k\eps,1,1,\ldots,1)$ from top to bottom,  which is increasing as required.  While this shows that labels of the form $1-k\eps$ with $k>0$ on $D$ do not cause any difficulty, it does not fulfill our promise to show that such labels never occur.  So
let $\bet^+$ be the element covering $\bet$ on $r$, and let $\bet^-$ be the element $\bet$ covers on $D$.  If $\bet \rightarrow\bet^-$ gets the label $1-k\eps$ with $k>0$, this means by definition of our label modifications that $\bet^- = \subperm{\bet^+}{1}{\ell-2} =  \subperm{\bet^+}{3}{\ell}$ where $\ell = |\bet^+|$.  Since $\alp$ only occurs in $\bet$ and hence $\bet^-$ as a prefix, the element covered by $\bet^-$ along $D$ is then $\subperm{\bet^+}{1}{\ell-3} = \subperm{\bet^+}{3}{\ell-1}$.  But $\subperm{\bet^+}{3}{\ell-1} = \subpermbet{3}{\ell-1}$, contradicting the fact that $\alp$ only occurs in $\bet$ as a prefix.  We conclude that all of the labels along $D$ are indeed 0 or 1.

Now suppose $[\alp,\bet]_r$ contains another increasing maximal chain $D'$.
Our approach will be to consider the possible locations of the occurrence of $\alp$ at which $D'$ ends, and to show some contradiction in each case.
If $D'$ ends at the rightmost occurrence $\subseqbet{i}{j}$ of $\alp$ in $\bet$, then $D'$ must equal $D$ to be increasing, since the path along $D'$ must make all its left deletions before making its right deletions.  Therefore, assume $D'$ ends at an occurrence of $\alp$ in $\bet$ that is not rightmost, namely $\subseqbet{i'}{j'}$. Since $D'$ is increasing, its last deletion is on the right and thus receives a label of the form $1-k\eps$ for $k\geq0$.  This already yields a contradiction in the case when $\sig=1$ since any permutation covering the permutation 1 is monotone.  Therefore, assume that $|\sig|>1$ for the rest of this proof.  Among those increasing maximal chains different from $D$, choose $D'$ so that $i'$ is as large as possible. There are three cases to consider.

First suppose $i'=i-1$ which, by Lemma~\ref{lem:monotone}, implies that $\subpermbet{i'}{j'+1}$ is monotone. Since the last deletion along $D'$ is on the right, the element covering $\alp$ on $D'$ is this monotone permutation $\subpermbet{i'}{j'+1}$.  But our convention for monotone permutations then implies that the last edge in $D'$ is labeled 0, which is a contradiction.

Next suppose that $i'=i-2$.  Thus $\subpermbet{i-2}{j-2} = \subpermbet{i}{j} = \alp$.  If it is also the case that $\subpermbet{i-1}{j-1} = \alp$, then applying Lemma~\ref{lem:monotone} three times tells us that $\subpermbet{i-2}{j}$ is monotone.  Since the last deletion of $D'$ is on the right, the element covering $\alp$ on $D'$ will be $\subpermbet{i-2}{j-1}$, obtaining a contradiction as in the previous paragraph.
Thus $\subpermbet{i-1}{j-1} \neq \alp$, and so $\alp$ straddles $\subpermbet{i-2}{j}$.  Since $D'$ ends at $\subpermbet{i-2}{j-2}$, the last two edges on $D'$ must receive the label $1-k\eps$ followed by the label $1-(k+1)\eps$ for some $k\geq0$.  This contradicts $D'$ having increasing labels.

For the third and final case, let $i' < i-2$.  Suppose first that, except for $\subseqbet{i}{j}$, the rightmost occurrence of $\alp$ in $\bet$ is  $\subseqbet{i'}{j'}$.  Then consider the subinterval $[\alp, \subpermbet{i'}{j}]$ of $[\alp,\bet]_r$.  Since $\alp$ straddles $\subpermbet{i'}{j}$, Theorem~\ref{thm:disconnected} implies that $[\alp,\bet]_r$ has a non-trivial disconnected subinterval, contradicting our hypothesis.  Thus there must be some other occurrence $\subseqbet{i''}{j''}$ of $\alp$ with $i' < i'' < i$ and, among all such occurrences, let us choose $i''$ to be as close to $i'$ as possible.  If $i''=i'+1$, then $\subpermbet{i'}{j'+1}$ is monotone, yielding a contradiction as before.  If $i''=i'+2$, then we can apply the argument of the previous paragraph with $i''$ and $j''$ in place of $i$ and $j$ to yield a contradiction.  Finally, suppose $i'' > i'+2$, and consider the subinterval $[\alp, \subpermbet{i'}{j''}]$ of $[\alp,\bet]_r$.  By our choice of $i''$, we see that $\alp$ straddles $\subpermbet{i'}{j''}$, yielding by Theorem~\ref{thm:disconnected} a non-trivial disconnected subinterval of $[\alp,\bet]_r$, contradicting our hypothesis.
\end{proof}

We conclude this section with a comparison of Theorem~\ref{thm:shellable} to related results in the literature.
\begin{remark} It follows from Theorem~\ref{thm:shellable} that the homotopy type of an interval $[\sigma,\tau]$ without non-trivial disconnected subintervals is a wedge of spheres.  As we mentioned in Subsection~\ref{sub:clbackground}, the number of spheres is $|\mu(\sigma,\tau)|$ so, by Theorem~\ref{thm:mobius}, we conclude that $\Delta(\sigma,\tau)$ is homotopic to a single sphere or is contractible.  This same conclusion is given as \cite[Theorem~2.8]{SaWi12} but they do not show shellability and instead use discrete Morse theory.  In fact, their conclusion applies to all intervals $[\sigma,\tau]$ in $\bigP$, even those containing disconnected subintervals.  An example of this more general setting is the 2-dimensional simplicial complex in Figure~\ref{fig:ordercomplex}, which contracts to a sphere of dimension 1.

It also follows from Theorem~\ref{thm:shellable} that the single spheres that arise must be of the same dimension as $\Delta(\sigma,\tau)$, but this is no longer true in the non-shellable cases, as shown by the aforementioned example in Figure~\ref{fig:ordercomplex}.  In summary, Theorem~\ref{thm:shellable} has a stronger hypothesis but also a stronger conclusion than \cite[Theorem~2.8]{SaWi12}.
\end{remark}

\begin{remark}
Billera and Myers \cite{BiMy98} have shown a general result of a similar flavor to Theorem~\ref{thm:shellable}: a bounded graded poset is shellable if it is $(2+2)$-free, meaning that it contains no induced subposet that is isomorphic to a disjoint sum of two chains of length 2.  In \cite{Wac99}, Wachs shows that such (2+2)-free posets are CL-shellable.  We note, however, that these results are not enough to imply Theorem~\ref{thm:shellable}, since there exist intervals in $\bigP$ that contain no non-trivial disconnected subintervals but fail to be (2+2)-free.   One example is in Figure~\ref{fig:labeling}, where the two chains $2143 < 21435$ and $2134 < 13245$ are of the form $2+2$.
\end{remark}

\section{All intervals are rank-unimodal and strongly Sperner}\label{sec:unimodal}

Since intervals in $\bigP$ are graded, it is natural to ask about the structure of these intervals with regard to the number of elements at each rank.  It is clear from Figure~\ref{fig:12-213546} that the intervals are not rank-symmetric in general.  It is also the case that the sequence of rank sizes is not log-concave in general.  For example, the interval $[1,1265473]$ has one element of rank 0 and five elements of rank 2 but only two elements of rank 1. However, there are two interesting properties that intervals of $\bigP$ have: they are rank-unimodal and strongly Sperner. Our goal in this section is to prove this assertion.  We note that neither result is known for the classical pattern poset \cite{McSt15}.  Even though the definitions of rank-unimodal and strongly Sperner are quite different, the two results appear together because they both rely heavily on the same injection between rank levels (Lemma~\ref{lem:injection}), and because Theorem~\ref{thm:griggs} below of Griggs connects the two properties.  

\subsection{Rank-unimodality}
For a finite graded poset $Q$, the set of elements of rank $i$ will be called a \emph{rank level} of $Q$, and the cardinality $a_i$ of this rank level will be referred to as the \emph{size} of rank $i$.
Recall that a finite graded poset of rank $N$ is called {\em rank-unimodal} if the sequence $a_0,a_1,\dots,a_N$ of rank sizes is unimodal. 

Before proving rank-unimodality for intervals in $\bigP$, it will be helpful and informative to examine an explicit 3-step method for constructing such intervals, up to isomorphism.  Fix $\sigma \leq \tau$, and let $n=|\tau|$. Consider the poset $P_1$ of intervals $[i,j]=\{i,i+1,\dots,j\}$ with $1\le i\le j\le n$, ordered by set inclusion. Each interval $[i,j]$ is associated with a subsequence $\subseqtau{i}{j}$ of $\tau$. The poset $P_1$ is a graded join-semilattice of rank $n-1$, having $n-r$ elements or rank $r$ for all $r$, namely $[i,i+r]$ for $1\le i\le n-r$.  See Figure~\ref{fig:L} for an example.

\begin{figure}[htbp]
\begin{center}
\begin{tikzpicture}[scale=0.8]
\tikzstyle{every node}=[rectangle, rounded corners=3pt, inner sep=3pt]
\draw (0,0) node[draw] (00) {[1,1]};
\draw (2,0) node[draw] (20) {[2,2]};
\draw (4,0) node[draw] (40) {[3,3]};
\draw (6,0) node[draw] (60) {[4,4]};
\draw (8,0) node[draw] (80) {[5,5]};
\draw (10,0) node[draw] (100) {[6,6]};
\draw (1,1) node[draw] (11) {[1,2]};
\draw[line width=1.5pt] (3,1) node[draw] (31) {[2,3]};
\draw[line width=1.5pt] (5,1) node[draw] (51) {[3,4]};
\draw (7,1) node[draw] (71) {[4,5]};
\draw[line width=1.5pt] (9,1) node[draw] (91) {[5,6]};
\draw[line width=1.5pt] (2,2) node[draw] (22) {[1,3]};
\draw[line width=1.5pt] (4,2) node[draw] (42) {[2,4]};
\draw[line width=1.5pt] (6,2) node[draw] (62) {[3,5]};
\draw[line width=1.5pt] (8,2) node[draw] (82) {[4,6]};
\draw[line width=1.5pt] (3,3) node[draw] (33) {[1,4]};
\draw[line width=1.5pt] (5,3) node[draw] (53) {[2,5]};
\draw[line width=1.5pt] (7,3) node[draw] (73) {[3,6]};
\draw[line width=1.5pt] (4,4) node[draw] (44) {[1,5]};
\draw[line width=1.5pt] (6,4) node[draw] (64) {[2,6]};
\draw[line width=1.5pt] (5,5) node[draw] (55) {[1,6]};
\draw (00) -- (11) -- (22) -- (33);
\draw (40) -- (31) -- (42) -- (53);
\draw (80) -- (71) -- (62) -- (73);
\draw (11) -- (20) -- (31) -- (22);
\draw (100) -- (91) -- (82) -- (73) -- (64) -- (55) -- (44) -- (33) -- (42) -- (51) -- (40);
\draw (44) -- (53) -- (62) -- (51) -- (60) -- (71) -- (82);
\draw (80) -- (91);
\draw (53) -- (64);
\end{tikzpicture}
\caption{For $\ist=[12,213546]$, the full poset above shows $P_1$, with $P_2$ given by the elements outlined with bold lines. Figure~\ref{fig:12-213546-injection} shows $P_3$.  In this example, the breaking rank is $b=1$ because $[4,6] \sim [1,3]$.}
\label{fig:L}
\end{center}
\end{figure}

For the second step, let $P_2$ be the poset obtained from $P_1$ by deleting all the elements $[i,j]$ such that $\subpermtau{i}{j}$ avoids $\sigma$. For the final step, let $P_3$ be the quotient\ poset\footnote{
There is more than one notion of quotient poset in the literature.  Our construction follows Hallam and Sagan \cite{HaSa15}: given an equivalence relation $\sim$ on a poset $P$, the \emph{quotient} $P/\sim$ is the set of equivalence classes, with $X \leq Y$ in $P/\sim$ if and only if $x \leq y$ in $P$ for some $x \in X$ and some $y\in Y$.  Hallam and Sagan show that $P/\sim$ is a poset if whenever $X \leq Y$ in $P/\sim$, we have that for all $y\in Y$ there exists $x\in X$ such that $x\leq y$ in $P$.  We see that $P_2/\sim$ satisfies this condition and so $P_3$ is indeed a poset.}  $P_2/\sim$, where $\sim$ is the equivalence relation defined by $[i,j]\sim[i',j']$ whenever $\subpermtau{i}{j}=\subpermtau{i'}{j'}$.
For equivalence classes in the quotient poset, denoted $\ol{[i,j]}$, recall that
one defines $\ol{[i,j]}\le\ol{[i',j']}$ if there exist $[i_1,j_1]\in\ol{[i,j]}$ and $[i_2,j_2]\in\ol{[i',j']}$ such that $[i_1,j_1] \subseteq [i_2,j_2]$.
It is clear that the map $\ol{[i,j]}\mapsto\subpermtau{i}{j}$ is a poset isomorphism between $P_3$ and the interval $[\sigma,\tau]$ in $\bigP$; see Figure~\ref{fig:12-213546-injection} for an example.  In what follows, we will work with $P_3$ instead of $\ist$ when helpful.

\begin{figure}[htbp]
\begin{center}
\begin{tikzpicture}[scale=1.3]
\tikzstyle{every node}=[rectangle, rounded corners=3pt, inner sep=3pt];
\draw (2,0) node[draw] (12) {$12\leftrightarrow\ol{[2,3]}$};
\draw (0,1) node[draw] (123) {$123\leftrightarrow\ol{[2,4]}$};
\draw (2,1) node[draw] (213) {$213\leftrightarrow\ol{[1,3]}$};
\draw (4,1) node[draw] (132) {$132\leftrightarrow\ol{[3,5]}$};
\draw (0,2) node[draw] (2134) {$2134\leftrightarrow\ol{[1,4]}$};
\draw (2,2) node[draw] (1243) {$1243\leftrightarrow\ol{[2,5]}$};
\draw (4,2) node[draw] (1324) {$1324\leftrightarrow\ol{[3,6]}$};
\draw (1,3) node[draw] (21354) {$21354\leftrightarrow\ol{[1,5]}$};
\draw (3,3) node[draw] (12435) {$12435\leftrightarrow\ol{[2,6]}$};
\draw (2,4) node[draw] (213546) {$213546\leftrightarrow\ol{[1,6]}$};
\draw (12) -- (132) -- (1324) -- (12435) -- (213546) -- (21354);
\draw (2134) -- (123) -- (12);
\draw (12) -- (213) -- (1324);
\draw[ultra thick] (213) -- (2134) -- (21354);
\draw[ultra thick] (132) -- (1243) -- (12435);
\draw (123) -- (1243) -- (21354);
\end{tikzpicture}
\caption{The interval $[12,213546]$ in $\bigP$ has already appeared in Figure~\ref{fig:12-213546}, but here it is adorned with the corresponding elements of $P_3$.  The bold edges show two examples of the injection of Lemma~\ref{lem:injection}. For the injection from $m_1=2$ elements of $E_1$ into $E_2$, we chose $k=2$ in the notation of the proof of Lemma~\ref{lem:injection}.  For the injection from $m_2=2$ elements of $E_2$ into $E_3$, our only valid option is $k=3$.}
\label{fig:12-213546-injection}
\end{center}
\end{figure}
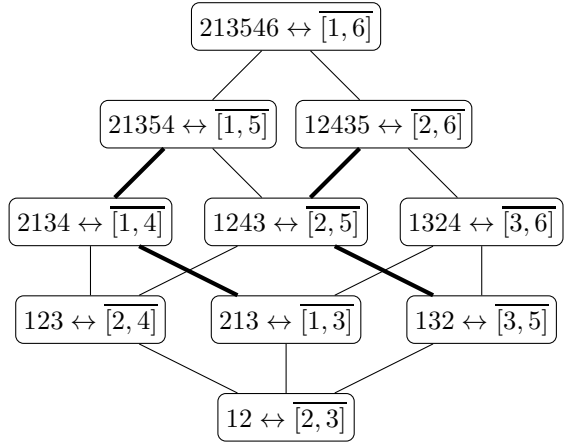

From this construction of $\ist$, we can make our first observation of this section about the structure of $\ist$: the upper ranks of $\ist$ are ``grid-like.''  More precisely, let $N=|\tau|-|\sigma|$ denote the rank of $\ist$ (or of $P_3$).  By comparing the structure of $P_1$ and $P_3$, we know there are at most $N+1-r$ elements of rank $r$ in $\ist$.  Let $b$ denote the largest value of $r$ such that the number of elements of rank $r$ is \emph{strictly less} than $N+1-r$.  We will call $b$ the \emph{breaking rank} of $\ist$.  In the context of $P_3$, $b$ is the rank of $\ol{[i,i+j]}$ where $j$ is the largest value such that either an element of the form $[i,i+j]$ was deleted when going from $P_1$ to $P_2$, or there is an equivalence class of the form $\ol{[i,i+j]}$ in $P_3$ containing more than one element of $P_2$.  For example, for $[12,213546]$ exhibited in Figures~\ref{fig:L} and~\ref{fig:12-213546-injection}, we have $b=1$ since $[4,6] \sim [1,3]$. Notice that these differences between $P_1$ and $P_3$ that happen at the breaking rank trickle down $P_3$, meaning that if $r$ is a rank less than the breaking rank then there will again be strictly less than the maximum $N+1-r$ elements of rank $r$.  

By the ``grid-like'' structure of the upper ranks of $\ist$
we mean that, by definition of the breaking rank $b$, the rank levels of rank $r$ with $b < r \leq N$ of $\ist$ are essentially the same as the top $N-b$ rank levels of $P_1$.  For example, in Figure~\ref{fig:L}, the grid-like form of the top three rank levels gives rise to the grid-like form of the top three rank levels in Figure~\ref{fig:12-213546-injection}. This observation about the upper ranks directly relates to the rank-unimodality of $\ist$ since it implies
\begin{equation}\label{eq:downward}
a_{b+1} > a_{b+2} > \dots > a_N,
\end{equation}
which will be the basis of the ``downward portion'' of the unimodality.

The following lemma will be used in both the proof of rank-unimodality and of the strongly Sperner property, and allows us to pair up elements of adjacent ranks.

\begin{lemma}\label{lem:injection}
In the interval $[\sigma,\tau]$ in $\bigP$, let $E_r$ denote the set of elements of rank $r$ for $0 \leq r \leq N$, where $N=|\tau|-|\sigma|$.  For $0 \leq r \leq N-1$, pick any $m_r \leq \min{\{|E_r|, N-r\}}$.  Then for any selection of $m_r$ elements of $E_r$, there exists an injection $f_r$ from these elements into $E_{r+1}$ with the property that $\pi < f_r(\pi)$ for each such element $\pi$.
\end{lemma}

\begin{proof}
We will work in the setting of $P_3$ as defined above, and the result will follow by the isomorphism with $\ist$.  Fix $m_r$ elements of rank $r$.  These elements of rank $r$ in $P_3$ take the form $\ol{[i, i+s]}$, where $s=|\sigma|-1+r$.  For each such element, choose the representative $[i,i+s]$ with smallest $i$.  Let $I_s \subset \{1,2,\ldots,n-s\}$ be the set of left endpoints of these $m_r$ representatives, where $n=|\tau|$ as before.  Since $m_r \leq N-r = n-s-1$, there is some $k \in \{1,2,\ldots,n-s\} \setminus I_s$.  Fix such a $k$, and define, for each $i\in I_s$,
\[
f(i)=\begin{cases} i & \text{ if }i<k,\\ i-1 & \text{ if }i>k. \end{cases}
\]
We claim that the map $\ol{[i,i+s]}\mapsto\ol{[f(i),f(i)+s+1]}$, defined for $i\in I_s$, is an injection from elements of rank $r$ to elements of rank $r+1$ in $P_3$.  See Figure~\ref{fig:12-213546-injection} for examples.

To see that this map is well defined, note first that $1\le f(i)\le n-s-1$ for all $i\in I_s$, and so $\subpermtau{f(i)}{f(i)+s+1}$ is defined.  Additionally, both $\subpermtau{i}{i+s+1}$ and $\subpermtau{i-1}{i+s}$ (when defined) contain $\subpermtau{i}{i+s}$, which translates to the desired property that $\pi < f_r(\pi)$ for all $\pi$.  Since $\subpermtau{i}{i+s}$ in turn contains $\sigma$, we get that $\ol{[f(i),f(i)+s+1]}$ is indeed an element of $P_3$.

Now we show that the map is injective.  Clearly, if $[i,i+s]\nsim [i',i'+s]$, then $[i-1,i+s]\nsim [i'-1,i'+s]$ and $[i,i+s+1]\nsim [i',i'+s+1]$.  Thus, the only possible way that injectivity could fail would be if there exist $i,i'\in I_s$ with $i<k<i'$ such that $[i,i+s+1]\sim[i'-1,i'+s]$.  But then $[i+1,i+s+1]\sim[i',i'+s]$, contradicting the choice of $i'\in I_s$ as the smallest left endpoint of a representative of the class.
\end{proof}

We can now deduce a combinatorial proof of the rank-unimodality of all intervals in $\bigP$.

\begin{corollary}\label{cor:rankunimodal}
Every interval $[\sigma,\tau]$ in $\bigP$ is rank-unimodal.
\end{corollary}

\begin{proof}
Since we have already shown \eqref{eq:downward}, it remains to show that $a_0 \leq a_1 \leq \cdots \leq a_b \leq a_{b+1}$.  To see this, fix $r$ with $0 \leq r \leq b$ and consider the elements of rank $r$ in $\ist$.  By definition of the breaking rank $b$, there are strictly less than $N+1-r$ elements of rank $r$.  The result now follows directly from the injection of Lemma~\ref{lem:injection}.
\end{proof}

\begin{remark}In the case of the classical pattern poset, it is an open question whether every interval is rank-unimodal \cite{McSt15}. It is known to be true for intervals $[\sigma,\tau]$ with $|\tau| \leq 8$.
\end{remark}

\subsection{Strongly Sperner property}
This subsection is devoted to a proof that all intervals in $\bigP$ are strongly Sperner.  Let us first give the necessary background.  Recall that a graded poset $Q$ is said to be \emph{Sperner} if the largest rank size equals the size of the largest antichain.  Since every rank level of $Q$ consists of an antichain, the Sperner property is equivalent to the condition that some rank level of $Q$ is an antichain of maximum size.  For any positive integer $k$, a \emph{$k$-family} of $Q$ is a union of $k$ antichains.  We are now ready for the definition of one of the well-known generalizations of the Sperner property.

\begin{definition}
Let $Q$ be a poset of rank $N$.  For an integer $k$ with $1 \leq k \leq N+1$, we say that $Q$ is \emph{$k$-Sperner} if the sum of the sizes of the $k$ largest ranks equals the size of the largest $k$-family.  In addition, $Q$ is said to be \emph{strongly Sperner} if it is $k$-Sperner for all $k$.  
\end{definition}

Again, being $k$-Sperner is equivalent to having a set of $k$ ranks that is a $k$-family of maximum size.  Our technique for showing that every interval of $\bigP$ is strongly Sperner will make heavy use of the injections $f_r$ from Lemma~\ref{lem:injection} and of Theorem \ref{thm:griggs} below, due to Griggs \cite{Gri80}, with shorter proofs appearing in \cite{GSS80,Lec97}.  

Letting $Q$ be a poset of rank $N$, for all $1\leq i \leq N+1$, let us denote the $i$th largest rank level by $L_i$ and its size by $\ell_i$, where we break ties in size arbitrarily.  For example, for the interval $[12,213546]$ of Figure~\ref{fig:12-213546-injection}, we have two choices for $L_1$, both of which give $\ell_1=3$, while $\ell_4=\ell_5=1$.  We will say that $Q$ is \emph{$i$-rank intersecting} if there exist $\ell_i$ disjoint chains which intersect each of the rank levels $L_1, L_2, \ldots, L_i$.  In particular, note that these disjoint chains must cover $L_i$.  In addition, we say that $Q$ is \emph{strongly rank intersecting} if it is $i$-rank intersecting for all $i$ with $1\leq i \leq N+1$.  

\begin{theorem}[\cite{Gri80}]\label{thm:griggs}
If a poset $Q$ is rank-unimodal, then $Q$ is strongly Sperner if and only if it is strongly rank intersecting.
\end{theorem}

We now have the tools necessary to prove the promised result.

\begin{theorem}\label{thm:sperner}
Every interval $[\sigma, \tau]$ in $\bigP$ is strongly Sperner.
\end{theorem}

\begin{proof}
By Theorem~\ref{thm:griggs}, it suffices to show that $Q=\ist$ is strongly rank intersecting, and throughout this proof we will follow the notation of the paragraph immediately preceding Theorem~\ref{thm:griggs}.  Thus, fixing $i$ with $1 \leq i \leq N+1$, consider the $i$ largest rank levels $L_1, L_2, \ldots, L_i$.  Since $\ist$ is rank-unimodal and we can break ties among rank sizes arbitrarily, we can assume that $L_1, L_2, \ldots, L_i$ form a set of consecutive rank levels in some order, with $L_i$ occurring at one of the two extremes of the consecutive sequence.  Let us suppose that these rank levels occupy ranks $r_1$ through $r_2$ in $\ist$.

As we will see, the injections $f_r$ of Lemma~\ref{lem:injection} will be exactly what we need to complete the proof.  Since we wish to construct $\ell_i$ disjoint chains, we will be picking $\ell_i$ elements of rank $r$ for all $r_1 \leq r < r_2$.  In order to apply the lemma, we need to show that $\ell_i \leq N-r$ for all such $r$, meaning that we need $\ell_i \leq N-r_2+1$.  By definition of $L_i$ and hence $\ell_i$, we have $\ell_i = \min\{a_{r_1}, a_{r_2}\}$, where $a_j$ again denotes the number of elements of rank $j$.  We already showed when considering the grid-like structure of the upper levels of $\ist$ that $a_r \leq N+1-r$.  Thus $\ell_i \leq a_{r_2} \leq N-r_2+1$, as required.

Since $\ell_i = \min\{a_{r_1}, a_{r_2}\}$, we can start by picking $\ell_i$ elements of rank $r_1$.  Applying $f_{r_1}$ to these elements, we construct $\ell_i$ disjoint chains of length 1 from rank $r_1$ to rank $r_1+1$.  Continuing in this fashion, applying $f_{r_1+1}, f_{r_1+2}, \ldots, f_{r_2-1}$, we construct $\ell_i$ disjoint chains that intersect each of the rank levels $r_1$ through $r_2$, as required.  See the bold edges in Figure~\ref{fig:12-213546-injection} for an example with $i=3$. 
\end{proof}

\begin{remark}
Generalizing the Sperner property in a different way, a poset is said to be \emph{strictly Sperner} if every antichain of maximum size is a rank level.  To see that this property does not hold for intervals in $\bigP$, consider the antichain $\{123, 1432\}$ in the interval $[12, 12543]$.
\end{remark}

We conclude this section with an open problem.  When $\sigma$ occurs just once in $\tau$, we know from Proposition~\ref{pro:occursonce} that $[\sigma, \tau]$ is a lattice.
 
\begin{problem}
Characterize those intervals $\ist$ in $\bigP$ that are lattices.  As usual, we would prefer our characterization to be in terms of simple conditions on $\sigma$ and $\tau$.
\end{problem}

For example, for $|\tau|=4$, $[1,\tau]$ is a lattice if and only if either $\subpermtau{1}{3}$ or $\subpermtau{2}{4}$ (or both) is monotone, since otherwise both will be upper bounds for 12 and 21.

\section{The exterior of a permutation}
\label{sec:exterior}

\subsection{The length of the exterior}\label{sub:lengthexterior}
As we have seen in the previous sections, particularly in Theorem~\ref{thm:mobius} and Proposition~\ref{prop:xtautau}, the exterior of a permutation plays an important role in the study of $\bigP$.  From an enumerative perspective, a natural problem is to study the distribution of the statistic {\em length of the exterior} on permutations. Table~\ref{tab:exterior} shows the number of permutations according to this statistic. The first column (after the one indexing $n$) counts permutations whose exterior has length 1. Permutations with this property are usually called \emph{non-overlapping permutations} in the literature. It was shown by B\'ona~\cite{Bona11} that the number of non-overlapping permutations of length $n$ is approximately $0.364 n!$, but no exact formula is known.
Next we show a simple divisibility property of these numbers.

\begin{table}[htb]
\begin{tabular}{r|rrrrrrrrr}
$n\backslash k$ & 1 & 2 & 3 & 4 & 5 & 6 & 7 & 8 & 9\\ \hline
2 & 2 &&&&&&&&\\
3 & 4 & 2 &&&&&&&\\
4 & 12 & 10 & 2 &&&&&&\\
5 & 48 & 58 & 12 & 2 &&&&&\\
6 & 280 & 306 & 118 & 14 & 2 &&&&\\
7 & 1864 & 2186 & 822 &150 & 16 & 2 &&&\\
8 & 14840 & 17034 & 6580 & 1660 & 186 & 18 & 2 &&\\
9 & 132276 & 154162 & 58854 & 15118 & 2222 & 226 & 20 & 2 &\\
10 & 1323504 & 1532574 & 588898 & 150388 & 30238 & 2904 & 270 & 22 & 2
\end{tabular}
\caption{The number of permutations $\tau\in\S_n$ with $|x(\tau)|=k$.}
\label{tab:exterior}
\end{table}

\begin{lemma}\label{lem:nonov4}
For every $n\ge3$, the number of non-overlapping permutations in $\S_n$ is divisible by 4.
\end{lemma}

\begin{proof}
If $\tau\in \S_n$ is non-overlapping, then so are its reversal, its complement, and its reverse-complement, and these four permutations are all different. Indeed, if two of these were the same, then $\tau$ would be equal to its reverse-complement, which would imply that $\tau$ starts with a descent if and only if it ends with a descent, and so $\tau$ would be overlapping. Thus, we have partitioned the set of non-overlapping permutations into disjoint sets of size 4.
\end{proof}

While the above property of the first column of Table~\ref{tab:exterior} is straightforward to prove, there seem to be other less obvious congruence relations. For example, the data in the second column suggests that the number of permutations in $\S_n$ with exterior of length 2 might always be congruent to 2 modulo 4.

The entries in the rightmost nonzero diagonal of Table~\ref{tab:exterior} count permutations $\tau$ with $|x(\tau)|=|\tau|-1$. It is clear that
the number of such permutations in $\S_n$ is $2$ for every $n\ge2$, since by Lemma~\ref{lem:monotone}, only the monotone permutations $12\dots n$ and $n\dots 21$ satisfy this condition.
The entries in the diagonal immediately below these $2$s count permutations $\tau$ with $|x(\tau)|=|\tau|-2$. In Subsection~\ref{sub:finding_disconnected} we described the structure of the intervals under such permutations. Next we enumerate these permutations.
\begin{lemma}
For every $n\ge4$, 
\[ 
|\{\tau\in\S_n:|x(\tau)|=n-2\}|=2n+2.
\]
\end{lemma}

\begin{proof}
Since the length of the exterior of any permutation is the same for its complement, we can count permutations with $\tau_1<\tau_2$ and multiply the final number by 2.  We have shown that every $\tau\in\S_n$ with $|x(\tau)|=n-2$ satisfies Equation~\eqref{eq:shift2}. It follows that $\tau_i<\tau_{i+1}$ for every odd $i$ and, since $\tau$ is not monotone (otherwise $|x(\tau)|=n-1$), that $\tau_i>\tau_{i+1}$ for every even $i$, so $\tau$ is alternating.

By Equation~\eqref{eq:shift2}, the relative order of $\tau_1$ and $\tau_3$ determines whether the sequence of valleys of $\tau$ is increasing or decreasing, and similarly the relative order of $\tau_2$ and $\tau_4$ determines it for the sequence of peaks.
When one of these sequences is increasing and the other decreasing, which happens when $\subpermtau{1}{4}$ equals $1423$ or $2314$, 
then $\tau$ is completely determined.

If both sequences are increasing, that is, $\subpermtau{1}{4}=1324$, then $\tau$ is determined once we choose the value of $\tau_2$ relative to the valleys of $\tau$, that is, for which index $2\le i\le \lfloor\frac{n+1}{2}\rfloor$ we have $\tau_{2i-1}<\tau_2<\tau_{2i+1}$ (where we define $\tau_{n+1}=\tau_{n+2}=\infty$ for convenience). This choice forces the relative order of all peaks with respect to the valleys. There are $\lfloor\frac{n-1}{2}\rfloor$ possibilities for $\tau$ in this case.

Similarly, if both sequences are decreasing, that is, $\subpermtau{1}{4}$ equals $2413$ or $3412$, then $\tau$ is determined once we choose the value of $\tau_1$ relative to the peaks of $\tau$, that is, for which index $1\le i\le \lfloor\frac{n}{2}\rfloor$ we have $\tau_{2i}>\tau_1>\tau_{2i+2}$ (where we define $\tau_{n+1}=\tau_{n+2}=-\infty$ for convenience).  There are $\lfloor\frac{n}{2}\rfloor$ possibilities for $\tau$ in this case, with just the case $i=1$ corresponding to $\subpermtau{1}{4} = 3412$.

In total, the number of $\tau\in\S_n$ with $|x(\tau)|=n-2$ is
$$2\left(2+\left\lfloor\frac{n-1}{2}\right\rfloor+\left\lfloor\frac{n}{2}\right\rfloor\right)=2n+2.$$
\end{proof}

We do not have formulas for the other entries of Table~\ref{tab:exterior}.

\begin{problem}
Find a formula for the entries of Table~\ref{tab:exterior}.
\end{problem}

\subsection{Asymptotic behavior of the length of the exterior}\label{sub:asymp_ext}
Even without having exact formulas, we can study the behavior of $|x(\tau)|$ as the length of $\tau$ goes to infinity. 
For given $n$, we will write $\P_n$ and $\E_n$ to denote the probability and the expectation of events involving $\tau\in\S_n$ chosen uniformly at random. In other words, $\tau$ is a random variable with uniform distribution over $\S_n$.
We also use the notation $g(n)=o(f(n))$ or $g(n)\ll f(n)$ to mean that $\lim_{n\to \infty}g(n)/f(n)=0$.

The main result in this subsection is a constant asymptotic upper bound on the expected length of the exterior of a permutation.  We show that, as the length of the permutation tends to infinity, the limit of this expected value exists and is between $e-1$ and $e$.
We start with a lemma that bounds the probability that the length of the exterior is large. A similar result appears in~\cite[Lemma 18]{Per13}.

\begin{lemma}\label{lem:xgem}
For $1\le m\le n-1$, $$\P_n(|x(\tau)|\ge m)\le \sum_{i=m}^{\lfloor n/2 \rfloor} \frac{1}{i!} + o(n^{-1}),$$
where the summation is defined to be zero when $m>\lfloor n/2 \rfloor$.
\end{lemma}

\begin{proof}
For $1\le i\le n-1$, let $B_i$ be the event ``$\tau$ has a bifix of length $i$,'' that is, $\subpermtau{1}{i}=\subpermtau{n-i+1}{n}$. By definition, $|x(\tau)|$ is the largest $i$ such that $B_i$ holds, and
$$\P_n(|x(\tau)|\ge m)=\P_n\left(\bigcup_{i=m}^{n-1}B_i\right)\le \sum_{i=m}^{n-1} \P_n(B_i).$$
When $i\le \lfloor n/2 \rfloor$,
$$\P_n(B_i)=\frac{1}{i!},$$ since among all the possible permutations of the last $i$ entries, $B_i$ holds for exactly one of them.

When $i>\lfloor n/2 \rfloor$, we have the upper bound
\begin{equation}\P_n(B_i)\le \frac{1}{(n-i)!^{\lfloor \frac{n}{n-i}\rfloor-1}}.\label{eq:Bi}\end{equation}
This is because if we break up $\tau$ into $\lfloor \frac{n}{n-i}\rfloor$ blocks of $n-i$ entries starting at the beginning: $\subpermtau{1}{n-i}, \subpermtau{n-i+1}{2(n-i)},\dots$,
then, for $B_i$ to hold, the entries in all of these blocks have to be in the same relative order. 
For $\lfloor n/2 \rfloor<i\le n-\lfloor\sqrt{n}\rfloor$, we have that $n-i\ge\lfloor\sqrt{n}\rfloor$ and $\lfloor \frac{n}{n-i}\rfloor\ge2$, and so expression~\eqref{eq:Bi} gives
$$\P_n(B_i)\le\frac{1}{\lfloor\sqrt{n}\rfloor!}.$$
For $n-\lfloor\sqrt{n}\rfloor< i\le n-2$, we have that $n-i\ge2$ and $\lfloor \frac{n}{n-i}\rfloor\ge\lfloor\sqrt{n}\rfloor$, and so
$$\P_n(B_i)\le\frac{1}{2^{\lfloor\sqrt{n}\rfloor-1}}.$$
When $i=n-1$, it is easy to see directly that $$\P_n(B_{n-1})=\frac{2}{n!},$$ since by Lemma~\ref{lem:monotone}, $B_{n-1}$ only holds for the two monotone permutations.

Combining the above bounds, we get
$$\P_n(|x(\tau)|\ge m)\le \sum_{i=m}^{\lfloor n/2 \rfloor} \frac{1}{i!} + \sum_{i=\lfloor n/2 \rfloor+1}^{n-\lfloor\sqrt{n}\rfloor}\frac{1}{\lfloor\sqrt{n}\rfloor!}+\sum_{i=n-\lfloor\sqrt{n}\rfloor+1}^{n-2}\frac{1}{2^{\lfloor\sqrt{n}\rfloor-1}}+\frac{2}{n!}.$$
Clearly, the terms other than the first summation approach 0 faster than $n^{-k}$ for any positive constant $k$, as $n$ tends to infinity.
\end{proof}

\begin{theorem}\label{thm:expectedxlength} The limit of $\E_n(|x(\tau)|)$ as $n$ goes to infinity exists and
$$e-1\le \lim_{n\to\infty} \E_n(|x(\tau)|) \le e.$$
\end{theorem}

\begin{proof}
By definition of expectation, 
\begin{align}\nonumber\E_n(|x(\tau)|)&=\sum_{i=1}^{\lfloor n/2 \rfloor} i\,\P_n(|x(\tau)|=i)+\sum_{\lfloor n/2 \rfloor+1}^{n-1} i\,\P_n(|x(\tau)|=i)\\
&=\sum_{m=1}^{\lfloor n/2 \rfloor} \P_n\left(m\le |x(\tau)|\le \lfloor n/2 \rfloor\right)\,+\,o(1),\label{eq:E2parts}
\end{align}
where we have bounded the second summation from above by 
$$n\,\P_n\left(|x(\tau)|\ge \lfloor n/2 \rfloor+1\right)=o(1),$$ 
using Lemma~\ref{lem:xgem}.

To prove that the limit exists, we will show that $\E_n(|x(\tau)|)$ behaves asymptotically similarly to an increasing sequence, and that it is bounded from above.
For $n_0$ with $1\leq n_0 < n$, consider the bijection 
$$\begin{array}{ccc}
\S_n & \longrightarrow & \S_{n_0}\times\S_{n-n_0}\times\displaystyle\binom{[n]}{n_0} \\
\tau & \mapsto & (\sigma\ ,\ \pi\ ,\ A) \end{array}$$ 
defined as follows.
Let $$\sigma=\subpermtau{1}{\lfloor n_0/2 \rfloor]\cup[n-\lceil n_0/2 \rceil+1,n}$$
be the permutation obtained by concatenating the first $\lfloor n_0/2 \rfloor$ and the last
$\lceil n_0/2 \rceil$ entries of $\tau$, let $A$ be the set of values of these entries before applying the reduction map $\red$, and let $$\pi=\subpermtau{\lfloor n_0/2 \rfloor+1}{n-\lceil n_0/2 \rceil}$$ be the reduction of the remaining entries.
If $|x(\sigma)|\le \lfloor n_0/2 \rfloor$, then clearly $|x(\tau)|\ge|x(\sigma)|$, since the bifix of $\sigma$ of length $|x(\sigma)|$ is also a bifix of $\tau$. This means that for any given $m$,
the bijection associates to every triple $(\sigma,\pi,A)$ with $m\le |x(\sigma)| \le \lfloor n_0/2 \rfloor$ a permutation $\tau$ with  $|x(\tau)|\ge m$. It follows that
$$|\{\sigma\in\S_{n_0}:m\le |x(\sigma)| \le \lfloor n_0/2 \rfloor\}|\,(n-n_0)!\,\binom{n}{n_0}\ \le\  |\{\tau\in\S_{n}: |x(\tau)| \ge m\}|.$$ Dividing both sides by $n!$, we get
$$
\P_{n_0}\left(m\le |x(\tau)|  \le  \lfloor n_0/2 \rfloor\right)  \le  \P_n(|x(\tau)| \ge m)
\ =\  \P_{n}\left(m\le |x(\tau)| \le \lfloor n/2 \rfloor\right)\,+\, o(n^{-1})
$$
by Lemma~\ref{lem:xgem}. Summing over $m$, we have
$$\sum_{m=1}^{\lfloor n_0/2 \rfloor}\P_{n_0}\left(m\le |x(\tau)| \le \lfloor n_0/2 \rfloor\right)\ \le\
\sum_{m=1}^{\lfloor n/2 \rfloor} \P_{n}\left(m\le |x(\tau)| \le \lfloor n/2 \rfloor\right) + o(1).$$
Using now Equation~\eqref{eq:E2parts}, it follows that 
\begin{equation}\label{eq:Eincreasing}
\E_{n_0}(|x(\tau)|)\le \E_n(|x(\tau)|)+o'(1),
\end{equation}
where here $o'(1)$ is an expression that tends to $0$ when both $n$ and $n_0$ go to infinity.

To show that the sequence $\E_n(|x(\tau)|)$ is bounded from above by a constant, we again use  Lemma~\ref{lem:xgem} to conclude that
\begin{multline*}
\E_n(|x(\tau)|)\ =\ \sum_{m=1}^{n-1} \P_n(|x(\tau)|\ge m)\ \le\ \sum_{m=1}^{\lfloor n/2 \rfloor} \sum_{i=m}^{\lfloor n/2 \rfloor} \frac{1}{i!} + o(1)\\
=\ \sum_{i=1}^{\lfloor n/2 \rfloor} \frac{i}{i!}+o(1)
\ \le\  \sum_{i=1}^{\infty} \frac{1}{(i-1)!}+o(1)\ =\ e+o(1).
\end{multline*}

The above inequality implies that $L\coloneqq\limsup_n \{\E_n(|x(\tau)|)\}$ exists, and that $L\le e$. Let us show that $L$ is in fact the limit of the sequence.  For any $\epsilon>0$, choose $n_0$ so that it satisfies the following two conditions.  First, take $n_0$ large enough so that for $n\ge n_0$, we have $\E_{n_0}(|x(\tau)|)\le \E_n(|x(\tau)|)+\epsilon/2$ (such an $n_0$ exists by Equation~\eqref{eq:Eincreasing}).  Secondly, pick $n_0$ so that $\E_{n_0}(|x(\tau)|)>L-\epsilon/2$ (such an $n_0$ exists by definition of $L$).
Then, for all $n\ge n_0$,
$$\E_n(|x(\tau)|) \ge \E_{n_0}(|x(\tau)|)-\frac{\epsilon}{2}>L-\epsilon.$$
This proves that $$\lim_{n\to\infty} \E_n(|x(\tau)|)=L.$$

To prove the lower bound, recall that $B_i$ is the event ``$\tau$ has a bifix of length $i$.''
Clearly $\P_n(|x(\tau)|\ge m)\ge \P_n(B_m)$, which equals $\frac{1}{m!}$ when $1\le m\le \lfloor n/2 \rfloor$.
Thus,
$$\E_n(|x(\tau)|)= \sum_{m=1}^{n-1}\P_n\left(|x(\tau)|\ge m\right) \ge \sum_{m=1}^{\lfloor n/2 \rfloor} \P_n(B_m)=\sum_{m=1}^{\lfloor n/2 \rfloor} \frac{1}{m!}.$$
The limit of the right-hand side as $n\to\infty$ is $e-1$.
\end{proof}

\begin{problem}
Find the exact value of $\lim_{n\to\infty} \E_n(|x(\tau)|)$.
\end{problem}

When $n=10$, the value of $\E_n(|x(\tau)|)$ is approximately 1.909. 
Some relatively crude computations (for each $n$ with $11 \leq n \leq 100$, we took a random sampling of 7! permutations) suggest that the value of the limit is between 1.9 and 1.92.

\begin{problem}
Find the limiting distribution of $|x(\tau)|$, that is, find $$\lim_{n\to\infty} \P_n(|x(\tau)|=k)$$ for all~$k$.
\end{problem}

It was shown by B\'ona~\cite{Bona11} that when $k=1$, the sequence $\P_n(|x(\tau)|=1)$ is decreasing and so its limit exists, and that
$0.3640981 \le \lim_{n\to\infty} \P_n(|x(\tau)|=1) \le 0.3640993$. Even though for fixed $k$ the sequences $\P_n(|x(\tau)|=k)$ are not monotone in general, it is possible to obtain recurrences similar to those in~\cite{Bona11} to show that the corresponding limits exist for all $k$ and are positive.

\subsection{Permutations with no carrier, and the M\"obius function of most intervals}\label{sub:carrier}

Recall that $\tau$ is said to have a carrier element if $x(\tau)\nleq i(\tau)$. Carrier elements, defined in~\cite{BFS11}, play a crucial role in determining the M\"obius function of intervals in $\bigP$, as we pointed out in Subsection~\ref{sub:mobiuspapers}.  A question that arises is how many permutations have a carrier element. In this subsection we study this question, and we use our findings to deduce that most intervals (in a precise sense that we will describe) have a zero M\"obius function. Since this subsection involves inequalities in both $\mathbb{R}$ and $\bigP$, we will use the symbol $\lep$ in the latter case as a distinguisher.

For $n$ from $2$ to $10$, the number of permutations $\tau\in\S_n$ with no carrier element,
that is, satisfying $x(\tau)\lep i(\tau)$, is given by the sequence 
\begin{equation}\label{seq:nocarrier}
0,4,12,84,548,4172,33984,315800,3213032,\dots\ .
\end{equation}

\begin{problem} Find a formula for the number of permutations with no carrier element.
\end{problem}

We point out that the displayed terms of the sequence~\eqref{seq:nocarrier} are all divisible by four, although we have not proved if this is the case for all terms.
Note that permutations with $n \ge3$ and no carrier element include non-overlapping permutations, since an exterior of length one is trivially contained in the interior, and that the number of such permutations is divisible by 4 by Lemma~\ref{lem:nonov4}.
From an asymptotic perspective, comparing the sequence~\eqref{seq:nocarrier} with $n!$ suggests that, as $n$ grows, most permutations in $\S_n$ have no carrier element. Next we show that this is indeed the case.

\begin{theorem}\label{thm:xi}
$$\lim_{n\to\infty} \P_{n} (x(\tau)\not\lep i(\tau)) =0.$$
\end{theorem}

\begin{proof} 
Let $n=|\tau|$ and $k=|x(\tau)|$. If $k<n/2$, let $\xx(\tau)=\subpermtau{k+1}{n-k}$ be the permutation obtained from $\tau$ by deleting the initial and final occurrences of $x(\tau)$. If $k\ge n/2$, define $\xx(\tau)$ to be the empty permutation (and let us add the empty permutation to $\bigP$ as its new bottom element for convenience).  Since $\xx(\tau)\lep i(\tau)$, we have
$$\P_n(x(\tau)\not\lep i(\tau)) \le \P_n(x(\tau)\not\lep \xx(\tau)).$$ For every $m$, conditioning on whether $|x(\tau)|\le m$ or not, we can write $ \P_n(x(\tau)\not\lep \xx(\tau))$ as a sum of two terms, the first being
\begin{align}
\nonumber &\P_n(|x(\tau)|\le m)\ \P_n\left(x(\tau)\not\lep\xx(\tau)\;\big|\; |x(\tau)|\le m\right)\\
\nonumber &\le\ \P_n(|x(\tau)|\le m)\ \P_n\left(\subpermtau{1}{m}\not\lep \subpermtau{m+1}{n-m}\; \big|\; |x(\tau)|\le m\right)\\
\label{eq:notinside} &\le\ \P_n \left(\subpermtau{1}{m}\not\lep \subpermtau{m+1}{n-m}\right), \end{align}
and the second being
\begin{align}
\nonumber & \P_n(|x(\tau)|> m)\ \P_n\left(x(\tau)\not\lep\xx(\tau)\;\big|\; |x(\tau)|> m\right)\ \\
\nonumber & \le\ \P_n(|x(\tau)|> m) \\
\nonumber & \le\  \sum_{i=m+1}^{\infty} \frac{1}{i!} + o(n^{-1}),
\end{align}
using Lemma~\ref{lem:xgem}.

To bound~\eqref{eq:notinside}, notice that $\subpermtau{1}{m}$ and $\subpermtau{m+1}{n-m}$ are independent random permutations, and
split $\subpermtau{m+1}{n-m}$ into $\lfloor\frac{n-2m}{m}\rfloor$ disjoint blocks of size $m$ (with possibly some leftover entries that do not belong to any of the blocks). The probability that each individual block is not an occurrence of $\subpermtau{1}{m}$ is $1-1/m!$. For 
$\subpermtau{1}{m}\not\lep \subpermtau{m+1}{n-m}$ to hold, none of the blocks can be an occurrence of 
$\subpermtau{1}{m}$. Since these are independent events, we have that 
$$\P_n \left(\subpermtau{1}{m}\not\lep \subpermtau{m+1}{n-m}\right)\le\left(1-\frac{1}{m!}\right)^{\lfloor\frac{n}{m}\rfloor-2}.$$

Combining all these bounds, we get
\begin{equation}\label{eq:bound0}\P_n(x(\tau)\not\lep i(\tau))\ \le\ \left(1-\frac{1}{m!}\right)^{\lfloor\frac{n}{m}\rfloor-2} + \sum_{i=m+1}^{\infty} \frac{1}{i!} + o(n^{-1})\end{equation}
for every $m$.

Finally, it is enough to choose $m$ to be a slowly-growing function of $n$ with $m(n)\to\infty$ as $n\to\infty$, such as $m=m(n)=\lfloor \log\log n \rfloor$. For such choice of $m$, writing
$$\left(1-\frac{1}{m!}\right)^{\frac{n}{m}}=\left[\left(1-\frac{1}{m!}\right)^{m!}\right]^{\frac{n}{m m!}},$$ 
the term inside the square brackets approaches $e^{-1}$ as $n\to\infty$. Using Stirling's formula, we see that
\begin{align}
(\log\log n)(\log\log n)! & \ll (\log\log n)^2\frac{(\log\log n)^{\log\log n}}{e^{\log\log n}}\\
\nonumber & \ll
e^{(\log\log\log n)(\log\log n)}  \ll e^{\log n}=n,
\end{align}
and so the exponent $\frac{n}{mm!}$ goes to infinity as $n$ grows. It follows that 
$$\left(1-\frac{1}{m!}\right)^{\lfloor\frac{n}{m}\rfloor-2}\le \left(1-\frac{1}{m!}\right)^{\frac{n}{m}-3}\longrightarrow0$$
as $n \to \infty$.
Clearly, the summation of $1/i!$ in \eqref{eq:bound0} tends to 0 as $m\to\infty$. Thus, the whole right-hand side of \eqref{eq:bound0} tends to 0 as $n\to\infty$, proving the statement.
\end{proof}

A consequence of Theorem~\ref{thm:xi} is that for any fixed $\sigma$, most intervals of the form $[\sigma,\tau]$ will have zero M\"obius function.

\begin{corollary}\label{cor:mu0}
For fixed $\sigma$, let $\P_n^\sigma (\mu(\sigma,\tau)=0)$ denote the probability that $\mu(\sigma,\tau)\linebreak=0$ when $\tau$ is chosen uniformly at random among permutations in $\S_n$ that contain~$\sigma$. Then, for every $\sigma$,
$$\lim_{n\to\infty} \P_n^\sigma (\mu(\sigma,\tau)=0) =1.$$
\end{corollary}

\begin{proof}
When $|\tau|-|\sigma|>2$, Theorem~\ref{thm:mobius} states that $\mu(\sigma,\tau)=0$ unless $\sigma\lep x(\tau)\not\lep i(\tau)$.
By Theorem~\ref{thm:xi}, the probability that $x(\tau)\not\lep i(\tau)$ when $\tau$ is chosen uniformly at random from $\S_n$ tends to 0 as $n\to\infty$. It follows that, for any given $\sigma$,
$$\lim_{n\to\infty} \P_n(\mu(\sigma,\tau)=0) =1.$$

To restrict to permutations $\tau$ containing $\sigma$, note that
$$\P_n^\sigma (\mu(\sigma,\tau)\neq0) = \P_n\left(\mu(\sigma,\tau)\neq0 \;\big|\; \sigma\lep\tau\right)
\le \frac{ \P_n(\mu(\sigma,\tau)\neq0)}{\P_n(\sigma\lep\tau)}.$$  By Lemma~\ref{lem:sigletau}, $\lim_{n\to\infty}\P_n(\sigma\lep\tau)=1$, so $\lim_{n\to\infty}\P_n^\sigma (\mu(\sigma,\tau)\neq0)=0$ as claimed.
\end{proof}

We point out that Bernini et al.~\cite{BFS11} give several conditions on $\sigma$ and $\tau$ that imply $\mu(\sigma,\tau)=0$. For example, they show that this happens when the first two entries of $\tau$ are not involved in any occurrence of $\sigma$. While the results in~\cite{BFS11} show that the probability that $\mu(\sigma,\tau)=0$ is large, their conditions depend on $\sigma$, and thus they cannot be used to prove that for fixed $\sigma$ this probability gets arbitrarily close to one as the length of $\tau$ increases, as we do in Corollary~\ref{cor:mu0}.

\section*{Acknowledgement} 
We thank anonymous referees for suggestions that improved the exposition.

\bibliographystyle{alpha}
\bibliography{consec_pattern_poset}

\end{document}